\documentclass[final,hidelinks]{siamart220329}

\usepackage{amsfonts,amssymb}
\usepackage{graphicx,float,enumerate}

\newsiamremark{remark}{Remark}

\headers{An explicit and symmetric EWI for NLSE}{W. Bao, and C. Wang}

\title{An explicit and symmetric exponential wave integrator 
	for the nonlinear Schr\"{o}dinger equation with low regularity potential and nonlinearity\thanks{Submitted to the editors DATE.
		\funding{The first author was partially supported by the Ministry of Education of Singapore under its AcRF Tier 2 funding MOE-T2EP20122-0002 (A-8000962-00-00).}}}

\author{Weizhu Bao\thanks{Department of Mathematics, National University of Singapore, Singapore 119076
		(\email{matbaowz@nus.edu.sg}, \email{e0546091@u.nus.edu})} 
	\and Chushan Wang\footnotemark[2]}

\usepackage{amsopn}

\newcommand{\R}{\mathbb{R}}
\newcommand{\C}{\mathbb{C}}
\newcommand{\Z}{\mathbb{Z}}
\newcommand{\vx}{\mathbf{x}}
\newcommand{\rmd}{\mathrm{d}}
\newcommand{\vep}{\varepsilon}
\newcommand{\vphi}{\varphi}
\newcommand{\psihn}[1]{\psi_N^{#1}}
\newcommand{\phiht}{\Phi_h}
\newcommand{\Phiht}{\mathbf{\Phi}_h}
\newcommand{\psin}[1]{\psi^{#1}}
\newcommand{\en}[1]{e^{#1}}
\newcommand{\ehn}[1]{e_N^{#1}}
\newcommand{\phit}{\Phi}
\newcommand{\Phit}{{\mathbf{\Phi}}}
\newcommand{\Psin}[1]{\mathbf{\Psi}^{#1}}
\newcommand{\Psihn}[1]{\mathbf{\Psi}_N^{#1}}
\newcommand{\Psitn}[1]{{\mathbf{\Psi}(t_{#1})}}
\newcommand{\En}[1]{\mathbf{e}^{#1}}
\newcommand{\Ehn}[1]{\mathbf{e}_N^{#1}}

\newcommand{\mcalL}[1]{{\mathcal{L}^{#1}}}

\newcommand{\vphis}{{\varphi_\text{s}}}
\newcommand{\vphic}{\varphi_\text{c}}
\newcommand{\Hop}{{\mathbf{H}}}
\newcommand{\Aop}[1]{\mathbf{A}(#1)}
\newcommand{\A}{\mathbf{A}}
\newcommand{\newt}{\zeta}

\DeclareMathOperator{\sinc}{sinc}

\begin{document}

\maketitle

\begin{abstract}
	We propose and analyze a novel symmetric Gautschi-type exponential wave integrator (sEWI) for the nonlinear Schr\"odinger equation (NLSE) with low regularity potential and typical power-type nonlinearity of the form $ |\psi|^{2\sigma}\psi $ with $ \psi $ being the wave function and $ \sigma > 0 $ being the exponent of the nonlinearity. 
	The sEWI is explicit and stable under a time step size restriction independent of the mesh size. 
	We rigorously establish error estimates of the sEWI under various regularity assumptions on potential and nonlinearity.  
	For ``good" potential and nonlinearity ($H^2$-potential and $\sigma \geq 1$), we establish an optimal second-order error bound in the $L^2$-norm. For low regularity potential and nonlinearity ($L^\infty$-potential and $\sigma > 0$), we obtain a first-order $L^2$-norm error bound accompanied with a uniform $H^2$-norm bound of the numerical solution. Moreover, adopting a new technique of \textit{regularity compensation oscillation} (RCO) to analyze error cancellation, for some non-resonant time steps, the optimal second-order $L^2$-norm error bound is proved under a weaker assumption on the nonlinearity: $\sigma \geq 1/2$. For all the cases, we also present corresponding fractional order error bounds in the $H^1$-norm, which is the natural norm in terms of energy. Extensive numerical results are reported to confirm our error estimates and to demonstrate the superiority of the sEWI, including much weaker regularity requirements on potential and nonlinearity, and excellent long-time behavior with near-conservation of mass and energy. 
\end{abstract}



\section{Introduction}
The nonlinear Schr\"odinger equation (NLSE) arises in various physical applications such as quantum physics and chemistry, Bose-Einstein condensation (BEC), laser beam propagation, plasma and particle physics \cite{review_2013,ESY,NLS,LHY}. 
In this paper, we consider the following NLSE on a bounded domain $ \Omega = \Pi_{i=1}^d (a_i, b_i) \subset \R^d $ ($d= 1, 2, 3$) equipped with periodic boundary condition
\begin{equation}\label{NLSE}
	\left\{
	\begin{aligned}
		&i \partial_t \psi(\vx, t) = -\Delta \psi(\vx, t) + V(\vx) \psi(\vx, t) + f(|\psi(\vx, t)|^2) \psi(\vx, t), \quad \vx \in \Omega, \  t>0, \\
		&\psi(\vx, 0) = \psi_0(\vx), \quad \vx \in \overline{\Omega},
	\end{aligned}
	\right.
\end{equation}
where $t$ is time, $\vx\in \R^d$ is the spatial coordinate, and $ \psi:=\psi(\vx, t) $ is a complex-valued wave function.  Here, $ V:=V(\vx)\in L^\infty(\Omega) $ is a given real-valued external potential, and $f$ is assumed to be the power-type nonlinearity given by 
\begin{equation}\label{eq:f}
	f(\rho) = \beta \rho^\sigma, \quad \rho:=|\psi|^2 \geq 0,
\end{equation}
where $\beta \in \R$ is a given constant and $\sigma > 0$ is the exponent of the nonlinearity. 

There are many important dynamical properties of the solution $\psi$ to the NLSE \cref{NLSE} \cite{Ant}. The NLSE \cref{NLSE} is time reversible or symmetric, i.e., it is unchanged under the change of variable in time as $ t \rightarrow -t $ and complex conjugating the equation. It is also time transverse or gauge invariant, i.e., the equation still holds under the transformation $V \rightarrow V + \alpha$ with $\alpha$ a given constant and $\psi \rightarrow \psi e^{-i \alpha t}$, which immediately implies that the density $\rho = |\psi|^2$ is unchanged. Moreover, the NLSE \cref{NLSE} conserves the mass 
\begin{equation}
	M(\psi(\cdot, t)) = \int_{\Omega} |\psi(\vx, t)|^2 \rmd \vx \equiv M(\psi_0), \quad t \geq 0,
\end{equation}
and the energy
\begin{align}
	E(\psi(\cdot, t)) = \int_{\Omega} \left[ |\nabla \psi(\vx, t)|^2 + V(\vx)|\psi(\vx, t)|^2 + F(|\psi(\vx, t)|^2) \right] \rmd \vx \equiv E(\psi_0), 
\end{align}
where the interaction energy density $F(\rho) = \frac{\beta}{\sigma+1} \rho^{\sigma+1}$.

When the nonlinearity is chosen as the cubic nonlinearity (i.e., $\sigma = 1$ in \cref{eq:f}) and the potential is chosen as the harmonic trapping potential (namely, $ V(\vx) = |\vx|^2/2 $), the NLSE \cref{NLSE} reduces to the cubic NLSE with smooth potential, which is also known as the Gross-Pitaevskii equation (GPE), especially in the context of BEC. In this case, both $V$ and $f$ are smooth functions.
While the NLSE with smooth potential and nonlinearity, such as the GPE, is prevalent, diverse physics applications require the incorporation of low regularity potential and/or nonlinearity, including discontinuous potential, disorder potential, and non-integer power nonlinearity (see \cite{bao2023_semi_smooth, bao2023_EWI,henning2017,zhao2021} and references therein for the applications). 

For the cubic NLSE with sufficiently smooth initial data, many accurate and efficient numerical methods have been proposed and analyzed in the last two decades, including the finite difference time domain (FDTD) method \cite{FD,bao2013,review_2013,Ant,henning2017}, the exponential wave integrator (EWI) \cite{bao2014,ExpInt,SymEWI,bao2023_EWI}, and the time-splitting method \cite{bao2003JCP,BBD,lubich2008,review_2013,schratz2016,Ant,splitting_low_reg,RCO_SE,su2022,bao2023_semi_smooth}. Among these methods, the Strang time-splitting method is widely used due to its efficient implementation and the preservation of many dynamical properties including the time symmetry, time transverse invariance, dispersion relation and mass conservation at the discrete level \cite{Ant}. Recently, some low regularity integrators (LRI) have also been designed for the cubic NLSE with low regularity initial data and with/without potential \cite{LRI,LRI_sinum,LRI_error,bronsard2022,tree1,tree2,tree3}. 

Most of the aforementioned numerical methods can be extended straightforwardly to solve the NLSE \cref{NLSE} with low regularity potential and/or nonlinearity, e.g., the FDTD method \cite{henning2017}, the time-splitting method \cite{ignat2011,choi2021,bao2023_semi_smooth,zhao2021,bao2023_improved}, the EWI \cite{bao2023_EWI} and the LRI \cite{zhao2021,tree2,bronsard2022} (different from some singular nonlinearities \cite{sinum2019,bao2019,bao2022,bao2022singular,wang2023}, where regularization may be needed). However, their performances deviate considerably when compared to the smooth setting. As demonstrated in \cite{bao2023_EWI}, compared with FDTD methods, time-splitting methods and LRIs, in the presence of low regularity potential and/or nonlinearity, the EWI outperforms all of them (at least at semi-discrete level). Nevertheless, the EWI presented in \cite{bao2023_EWI} remains a first-order method and fails to preserve either the time symmetry or the conservation of mass and energy. It is of great interest to devise numerical schemes of high order that are capable of preserving the dynamical properties at discrete level. Based on the discussion above, we focus on designing and analyzing structure-preserving EWIs, especially time symmetric EWIs.  

In fact, extensive efforts have been made in the literature to design symmetric exponential integrators for the cubic NLSE, motivated by the structure-preserving properties and favorable long-time behavior of these symmetric schemes \cite{SymEWI,feng_dirac,Haier,bronsard2023,feng2023longtime,tree3}. However, most of them are fully implicit, which necessitates solving a fully nonlinear system at each time step. Usually, the nonlinear system is solved by a fixed point iteration, which is time-consuming, especially in two dimensions (2D) and three dimensions (3D) \cite{SymEWI,bronsard2023,feng2023longtime}. Actually, if the nonlinear system is not solved up to sufficient accuracy (e.g., machine accuracy), the symmetric property of the scheme may be destroyed. Moreover, although there are some explicit symmetric schemes in the literature \cite{feng_dirac}, they are only conditionally stable under certain CFL-type time step size restrictions. Additionally, all those schemes lack error estimates under low regularity assumptions on potential and nonlinearity. 

In this work, we introduce a second-order explicit and symmetric Gautschi-type EWI (sEWI), which possesses two main advantages: (i) it is explicit and symmetric while, remains stable under a time step size restriction independent of the mesh size; (ii) it is rigorously proved to be advantageous for the NLSE with low regularity potential and nonlinearity. Note that, despite sharing a similar spirit in the construction, the sEWI is a completely new method and it differs from the first-order EWI introduced in \cite{bao2023_EWI} in that it is a symmetric two-step method and is of second order. In fact, any symmetric method is at least second order \cite{Haier}. 



To show the merit of our new sEWI for the NLSE with low regularity potential and nonlinearity, we rigorously establish error estimates under varying regularity assumptions on potential and nonlinearity. The analysis is carried out in the one dimensional setting; however, it remains essentially the same in 2D and 3D. Main results are stated in Sections \ref{sec:main_1} and \ref{sec:main_2}, and we also summarize them here:
\begin{enumerate}[(i)]
	\item under the assumption of $H^2$-potential, $\sigma \geq 1$ and $H^4$-solution, we derive an optimal second-order error bound in $L^2$-norm at both semi-discrete (\cref{thm:second_order_semi}) and fully discrete (\cref{thm:second_order_full}) level; 
	\item under more general assumptions of $L^\infty$-potential, $\sigma > 0$ and $H^2$-solution, we establish a first-order error bound in $L^2$-norm along with a uniform $H^2$-bound of the numerical solution at both semi-discrete (\cref{thm:first_order_semi}) and fully discrete (\cref{thm:first_order_full}) level; 
	\item for the full-discretization scheme, the assumption on the nonlinearity in (i) can be relaxed to $\sigma \geq 1/2$ for certain non-resonant time steps (\cref{thm:second_order_improved}). 
\end{enumerate}
For first-order $L^2$-norm error bound in time, (ii) mirrors the first-order EWI in \cite{bao2023_EWI}, which, to our best knowledge, requires the weakest regularity of both potential and nonlinearity, among FDTD method \cite{henning2017}, Lie-Trotter time-splitting method \cite{bao2023_semi_smooth,bao2023_improved,ignat2011,choi2021}, and LRIs \cite{bronsard2022,tree2}. In terms of the optimal second-order error bound, the assumption of the sEWI in (i) is significantly weaker than the requirement of the Strang time-splitting method which requires $H^4$-potential and $\sigma \geq 3/2$ (or $\sigma = 1$) \cite{BBD,lubich2008,review_2013}. 
Although recent results in \cite{bao2023_improved} show that, for the full-discretization of the Strang time-splitting method, it can obtain optimal $L^2$-norm error bound under the same assumption as the sEWI in (i), a CFL-type time step size restriction is necessary. Actually, if imposing the same time step size restriction, the regularity requirement on nonlinearity for the sEWI can be further relaxed as in (iii). It should be noted the concept of ``low regularity" includes two aspects: one is, given low regularity potential and nonlinearity, determining the highest convergence order that can be achieved (corresponding to (ii)); and the other is identifying the lowest regularity required to obtain a given convergence order (corresponding to (i) and (iii)).



The rest of the paper is organized as follows. In Section 2, we present the sEWI and its spatial discretization by the Fourier spectral method. Sections 3-4 are devoted to the error estimates of the sEWI at semi-discrete level and fully discrete level, respectively. Numerical results are reported in Section 5 to confirm our error estimates. Finally, some conclusions are drawn in Section 6. Throughout the paper, we adopt standard notations of vector-valued Sobolev spaces as well as their corresponding norms. The bold letters are always used for vector-valued functions or operators. We denote by $ C $ a generic positive constant independent of the mesh size $ h $ and time step size $ \tau $, and by $ C(\alpha) $ a generic positive constant depending only on the parameter $ \alpha $. The notation $ A \lesssim B $ is used to represent that there exists a generic constant $ C>0 $, such that $ |A| \leq CB $.

\section{An explicit and symmetric exponential wave integrator}
In this section, we introduce an explicit and  symmetric exponential wave integrator (sEWI) and its spatial discretization to solve the NLSE \eqref{NLSE}. For simplicity of the presentation and to avoid heavy notations, we only carry out the analysis in one dimension (1D) and take $ \Omega = (a, b) $. Generalizations to 2D and 3D are straightforward.
In fact, main dimension sensitive estimates are the Sobolev embeddings. Throughout this paper, we only use embeddings $H^1(\Omega) \hookrightarrow L^4(\Omega)$,  $H^\frac{7}{4}(\Omega) \hookrightarrow L^\infty(\Omega)$ and $H^2(\Omega) \hookrightarrow L^\infty(\Omega)$, which are valid in 1D, 2D and 3D.
We define the periodic Sobolev spaces as (see, e.g., \cite{bronsard2022}, for the definition in the phase space)
\begin{equation*}
	H^m_\text{per}(\Omega):=\{\phi \in H^m(\Omega): \phi^{(k)}(a) = \phi^{(k)}(b), \quad k = 0 , \cdots, m-1\}, \quad m \in \Z^+. 
\end{equation*}
\subsection{A semi-discretization in time}
For simplicity, define an operator $B$ as 
\begin{equation}\label{eq:B_def}
	B(v)(x) := V(x)v(x) + f(|v(x)|^2)v(x), \quad v \in L^2(\Omega). 
\end{equation}
Choose a time step size $ \tau > 0 $ and denote time steps as $ t_n = n\tau $ for $ n = 0, 1, \cdots $. In the following, we shall always abbreviate $\psi(\cdot, t)$ by $\psi(t)$ for simplicity when there is no confusion. By Duhamel's formula, the exact solution $\psi$ satisfies
\begin{equation}\label{eq:duhamel_exact}
	\psi(t_n + \newt) = e^{i\newt\Delta} \psi(t_n) - i \int_0^\newt e^{i(\newt -s )\Delta} B(\psi(t_n + s)) \rmd s, \quad \newt \in \R. 
\end{equation}
Multiplying $e^{-i\newt\Delta}$ on both sides of \cref{eq:duhamel_exact}, we obtain
\begin{equation}\label{eq:duhamel_modified}
	e^{-i\newt\Delta} \psi(t_n + \newt) = \psi(t_n) - i \int_0^\newt e^{-is \Delta} B(\psi(t_n + s)) \rmd s. 
\end{equation}
Taking $\newt = \tau$ in \cref{eq:duhamel_modified}, we get
\begin{equation}\label{eq:exact_plus}
	e^{-i\tau\Delta} \psi(t_{n+1}) = \psi(t_n) - i \int_0^\tau e^{-is \Delta} B(\psi(t_n + s)) \rmd s. 
\end{equation}
Taking $\newt = -\tau$ in \cref{eq:duhamel_modified}, we get
\begin{align}\label{eq:exact_minus}
	e^{i\tau\Delta} \psi(t_{n-1}) &= \psi(t_n) - i \int_0^{-\tau} e^{-is \Delta} B(\psi(t_n + s)) \rmd s \notag \\
	&= \psi(t_n) + i \int_{-\tau}^0 e^{-is \Delta} B(\psi(t_n + s)) \rmd s. 
\end{align}
Subtracting \cref{eq:exact_minus} from \cref{eq:exact_plus}, we obtain
\begin{equation}\label{eq:exact_twisted}
	e^{-i \tau \Delta} \psi(t_{n+1}) - e^{i\tau\Delta} \psi(t_{n-1}) = -i \int_{-\tau}^{\tau} e^{-is\Delta} B(\psi(t_n + s)) \rmd s. 
\end{equation}
Multiplying $e^{i\tau\Delta}$ on both sides of \cref{eq:exact_twisted} yields
\begin{equation}\label{eq:exact}
	\psi(t_{n+1}) = e^{2i\tau\Delta} \psi(t_{n-1}) -i \int_{-\tau}^{\tau} e^{i(\tau - s)\Delta} B(\psi(t_n + s)) \rmd s. 
\end{equation}
Applying the Gautschi-type rule to approximate the integral in \cref{eq:exact}, i.e., approximating $\psi(t_n + s)$ by $\psi(t_n)$ and integrating out $e^{i(\tau - s)\Delta}$ exactly, we get
\begin{align}
	\int_{-\tau}^{\tau} e^{i(\tau - s)\Delta} B(\psi(t_n + s)) \rmd s 
	&\approx \int_{-\tau}^{\tau} e^{i(\tau - s)\Delta} B(\psi(t_n)) \rmd s \notag \\
	&= 2 \tau e^{i\tau\Delta} \vphis(\tau \Delta) B(\psi(t_n)),  \label{eq:approximation}
\end{align}
where $\vphis:\R \rightarrow \R$ is an analytic and even function defined as
\begin{equation}\label{eq:vphis_def}
	\vphis(\theta) = \sinc(\theta) = \frac{\sin(\theta)}{\theta}, \qquad \theta \in \R. 
\end{equation}
Plugging \cref{eq:approximation} into \cref{eq:exact} yields
\begin{equation}\label{eq:approx_ngeq1}
	\psi(t_{n+1}) \approx e^{2i\tau\Delta} \psi(t_{n-1}) -2i \tau e^{i \tau \Delta} \vphis(\tau \Delta) B(\psi(t_n)). 
\end{equation}
For the first step, we apply a first-order approximation as 
\begin{align}\label{eq:approx_neq0}
	\psi(t_1) 
	&= e^{i\tau\Delta}\psi(t_0) - i \tau \int_0^\tau e^{i(\tau - s)\Delta} B(\psi(t_0 + s)) \rmd s \notag \\
	&\approx e^{i\tau\Delta} \psi(t_0) - i \int_0^\tau e^{i(\tau - s)\Delta} B(\psi(t_0 )) \rmd s = e^{i\tau\Delta} \psi(t_0) - i \tau \vphi_1(i \tau \Delta) B(\psi(t_0)), 
\end{align}
where $ \vphi_1:\C \rightarrow \C $ is an entire function defined as
\begin{equation}
	\vphi_1(z) = \frac{e^z - 1}{z}, \qquad z \in \C. 
\end{equation}
The operators $\vphi_1(i \tau \Delta)$ and $\vphi_s(\tau \Delta)$ are defined through their action in the Fourier space, and they satisfy $ \vphi_1(i \tau \Delta) v, \vphi_s(\tau \Delta) v \in H^2_\text{per}(\Omega) $ for all $ v \in L^2(\Omega)$ (see (2.3) and (2.5) in \cite{bao2023_EWI}). 

Let $ \psin{n}(\cdot) $ be the approximation of $ \psi(\cdot, t_n) $ for $ n \geq 0 $. Applying the approximations \cref{eq:approx_ngeq1,eq:approx_neq0}, we obtain the sEWI as
\begin{equation}\label{eq:sEWI_scheme}
	\begin{aligned}
		\psin{n+1} &= e^{2i\tau \Delta} \psin{n-1} - 2i\tau e^{i\tau\Delta}\vphis(\tau \Delta) B(\psin{n}), \quad n \geq 1, \\
		\psin{1} &= e^{i\tau\Delta}\psin{0} - i\tau \vphi_1(i\tau\Delta) B(\psin{0}), \\
		\psin{0} &= \psi_0. 
	\end{aligned}
\end{equation}
One can check that the scheme is unchanged for $n \geq 1$ when exchanging $n-1$ and $n+1$ and replacing $\tau$ with $-\tau$, which implies that the scheme \cref{eq:sEWI_scheme} is symmetric in time.  	

\begin{remark}\label{rem:first_step}
	For the computation of the first step $\psin{1}$, instead of the first-order EWI we present here, one may choose other methods such as the Strang time-splitting method to guarantee time symmetry in the first step, which would be beneficial in the smooth setting. However, according to the analysis in \cite{bao2023_semi_smooth,bao2023_EWI,bao2023_improved}, time-splitting methods cannot give accurate approximation in the presence of low regularity potential and/or nonlinearity, and it will totally destroy the regularity of the numerical solution even for one step. Hence, if one wants a better approximation of the first step under the low regularity setting, we suggest computing $\psin{1}$ ``accurately" by using the first-order EWI a few times with a small time step. 
\end{remark}

\subsection{A full-discretization by using the Fourier spectral method in space}
In this subsection, we further discretize the sEWI \cref{eq:sEWI_scheme} in space by the Fourier spectral method to obtain a fully discrete scheme. Usually, the Fourier pseudospectral method is used for spatial discretization, which can be efficiently implemented with FFT. However, due to the low regularity of potential and/or nonlinearity, it is very hard to establish error estimates of the Fourier pseudospectral method, and it is impossible to obtain optimal error bounds in space as order reduction can be observed numerically \cite{bao2023_EWI,henning2022}. 

Choose a mesh size $ h=(b-a)/N $ with $ N $ being a positive even integer. Define the index set
\begin{equation*}
	\mathcal{T}_N = \left\{-\frac{N}{2}, \cdots, \frac{N}{2}-1 \right\}, 
\end{equation*}
and denote
\begin{equation}
	\begin{aligned}
		&X_N = \text{span}\left\{e^{i \mu_l(x - a)}: l \in \mathcal{T}_N\right\}, \quad \mu_l = \frac{2 \pi l}{b-a}. 
	\end{aligned}
\end{equation}
Let $ P_N:L^2(\Omega) \rightarrow X_N $ be the standard $ L^2 $-projection onto $ X_N $ as
\begin{equation}
	(P_N u)(x) = \sum_{l \in \mathcal{T}_N} \widehat u_l e^{i \mu_l(x - a)}, \qquad x \in \overline{\Omega} = [a, b],
\end{equation}
where $ u \in L^2(\Omega) $ and $\widehat u_l$ are the Fourier coefficients of $u$ defined as
\begin{equation}\label{eq:ft}
	\widehat{u}_l = \frac{1}{b-a} \int_a^b u(x) e^{-i \mu_l (x-a)} \rmd x,  \qquad l \in \Z.
\end{equation}

Let $ \psihn{n}(\cdot) $ be the numerical approximations of $ \psi(\cdot, t_n) $ for $ n \geq 0 $. Then the sEWI Fourier spectral method (sEWI-FS) reads 
\begin{equation}\label{eq:sEWI-FS_scheme}
	\begin{aligned}
		\widehat{(\psihn{n+1})}_l &= e^{- 2 i \tau \mu_l^2} \widehat{(\psihn{n-1})}_l - 2i\tau e^{- i \tau \mu_l^2} \vphi_s(\tau \mu_l^2) \widehat{(B(\psihn{n}))}_l, \quad l \in \mathcal{T}_N, \quad n \geq 1, \\
		\widehat{(\psihn{1})}_l &= e^{- i \tau \mu_l^2}\widehat{(\psihn{0})}_l - i\tau \vphi_1(-i\tau\mu_l^2) \widehat{(B(\psihn{0}))}_l, \quad l \in \mathcal{T}_N, \\
		\widehat{(\psihn{0})}_l &= \widehat{(\psi_0)}_l, \quad l \in \mathcal{T}_N.  
	\end{aligned}
\end{equation}	
Then the numerical solution $\psihn{n} \  (n \geq 0) \in X_N$ obtained from \cref{eq:sEWI-FS_scheme} satisfies
\begin{equation}\label{eq:sEWI-FS}
	\begin{aligned}
		\psihn{n+1} &= e^{2i \tau \Delta} \psihn{n-1} - 2i\tau e^{i \tau \Delta} \vphi_s(\tau \Delta) P_N B(\psihn{n}), \quad n \geq 1, \\
		\psihn{1} &= e^{i \tau \Delta}\psihn{0} - i\tau \vphi_1(i\tau\Delta) P_N B(\psihn{0}), \\
		\psihn{0} &= P_N \psi_0. 
	\end{aligned}
\end{equation}
One may check again that when exchanging $n+1$ and $n-1$ and replacing $\tau$ with $-\tau$ in the first equation of \cref{eq:sEWI-FS}, the equation remains unchanged, which implies that the fully discrete scheme is still time symmetric. Note that the Fourier spectral discretization of the Strang time-splitting method (which is considered in \cite{bao2023_improved}) is no longer symmetric in time. 

By performing the standard Von Neumann analysis, we obtain the following linear stability of the sEWI-FS method: Assuming that $V(x) \equiv V_0 \in \R$ and $f(\rho) \equiv f_0 \in \R $, then the sEWI-FS method \cref{eq:sEWI-FS_scheme} is stable when $\tau |V_0 + f_0| \leq 1$. In particular, this stability condition suggests that one shall chose the time step size $\tau$ satisfying $\tau \lesssim 1/(\sup_{x \in \Omega} |V(x)| ) $ in practice. 

\section{Error estimates of the semi-discretization \cref{eq:sEWI_scheme}} \label{sec:3}
In this section, we shall prove error estimates for the semi-discrete scheme \cref{eq:sEWI_scheme} under different regularity assumptions on the potential, nonlinearity and exact solution as shown below. 

Let $T_\text{max}$ be the maximal existing time of the solution to \cref{NLSE} and take $0<T<T_\text{max}$ be some fixed time. We shall work under the following three sets of assumptions: 

(A) Assumptions for ``good" potential and nonlinearity
	\begin{equation}\label{eq:A2}
		\begin{aligned}
			&V \in H^2_\text{per}(\Omega), \quad  \sigma \geq 1, \\
			&\psi \in C([0, T]; H^4_\text{per}(\Omega)) \cap C^1([0, T]; H^2(\Omega)) \cap C^2([0, T]; L^2(\Omega)). 
		\end{aligned}
	\end{equation}

(B) Assumptions for low regularity potential and/or nonlinearity
	\begin{equation}\label{eq:A1}
		V \in L^\infty(\Omega), \quad \sigma > 0, \quad \psi \in C([0, T]; H^2_\text{per}(\Omega)) \cap C^1([0, T]; L^2(\Omega)). 
	\end{equation}

(C) Assumptions with relaxed requirement on nonlinearity compared to \cref{eq:A2}
	\begin{equation}\label{eq:A3}
		\begin{aligned}
			&V \in H^2_\text{per}(\Omega), \quad  \sigma \geq 1/2, \\
			&\psi \in C([0, T]; H^4_\text{per}(\Omega)) \cap C^1([0, T]; H^2(\Omega)) \cap C^2([0, T]; L^2(\Omega)). 
		\end{aligned}
	\end{equation}

\subsection{Main results}\label{sec:main_1}
In this subsection, we present our main error estimates for the semi-discrete scheme under assumptions \cref{eq:A2,eq:A1}, respectively. Since optimal error bounds under \cref{eq:A3} can only be proved at the fully discrete level with non-resonant time steps, the corresponding result is postponed to Section 4. 

Let $\psin{n} ( n \geq 0)$ be obtained from the semi-discretization \cref{eq:sEWI_scheme}. For sufficiently smooth potential and nonlinearity as well as the exact solution, we have the following optimal second-order $L^2$-norm error bound.
\begin{theorem}[Optimal error bounds for ``good" potential and nonlinearity]\label{thm:second_order_semi}
	Under the assumptions \cref{eq:A2}, there exists $\tau_0>0$ sufficiently small such that when $0 < \tau < \tau_0$, we have
	\begin{equation}
		\| \psi(\cdot, t_n) - \psin{n} \|_{L^2} \lesssim \tau^2, \quad \| \psi(\cdot, t_n) - \psin{n} \|_{H^1} \lesssim \tau^\frac{3}{2}, \quad 0 \leq  n \leq T/\tau. 
	\end{equation}
\end{theorem}

For more general low regularity potential and/or nonlinearity, which inevitably results in low regularity of the exact solution, we have the following first-order $L^2$-norm error bound. 
\begin{theorem}[Error bounds for low regularity potential and/or nonlinearity]\label{thm:first_order_semi}
	Under the assumptions \cref{eq:A1}, there exists $\tau_0>0$ sufficiently small such that when $0 < \tau < \tau_0$, we have
	\begin{equation}
			\| \psi(\cdot, t_n) - \psin{n}  \|_{L^2} \lesssim \tau, \ \| \psi(\cdot, t_n) - \psin{n}  \|_{H^1} \lesssim \tau^\frac{1}{2},\ \| \psin{n} \|_{H^2} \lesssim 1, 
			 \quad 0 \leq n \leq T/\tau.
	\end{equation}
\end{theorem}

\begin{remark}\label{rem:semi}
	Recall the results in \cite{BBD,lubich2008,review_2013}, to obtain the second-order $L^2$-norm error bound of the Strang time-splitting method at semi-discrete level, it is required that $V \in H^4_\text{per}(\Omega)$ and $\sigma \geq 3/2$ (or $\sigma = 1$), which is much stronger than the requirement in Theorem \ref{thm:second_order_semi} of the sEWI. Similarly, in terms of the first-order $L^2$-norm error bound at semi-discrete level, as is already discussed in \cite{bao2023_EWI}, the sEWI also needs significantly weaker regularity than the Lie-Trotter splitting method analyzed in \cite{bao2023_semi_smooth}. 
\end{remark}

\begin{remark}
	Regarding the assumption of \cref{thm:first_order_semi}, according to the results in \cite{cazenave2003,kato1987} (see also \cite{bao2023_semi_smooth,bao2023_EWI,bao2023_improved}), when $V \in L^\infty(\Omega)$ and $\sigma > 0$ ($\sigma<2$ if in 3D), it can be expected that the unique exact solution $\psi \in C([0, T]; H^2_\text{per}(\Omega)) \cap C^1([0, T]; L^2(\Omega))$ if $\psi_0 \in H^2_\text{per}(\Omega)$ in the NLSE \cref{NLSE}. The assumption of \cref{thm:second_order_semi} is compatible with this assumption in the sense that the increment of differentiability orders of potential and nonlinearity are the same as the exact solution. In other words, when $V \in H^2_\text{\rm per}(\Omega)$ and $\sigma \geq 1$, if $\psi_0 \in H^4_\text{per}(\Omega)$, one could expect $\psi \in C([0, T]; H_\text{\rm per}^4(\Omega)) \cap C^1([0, T]; H^2(\Omega)) \cap C^2([0, T]; L^2(\Omega)) $. 
\end{remark}

In the rest of this section, we shall prove \cref{thm:second_order_semi,thm:first_order_semi}. Note that the assumption of \cref{thm:first_order_semi} is more general than that of \cref{thm:second_order_semi}. Hence, we start with the proof of \cref{thm:first_order_semi}, and the uniform $H^2$-bound of the numerical solution established in \cref{thm:first_order_semi} will be used to simplify the proof of \cref{thm:second_order_semi}. 

\subsection{Proof of \cref{thm:first_order_semi} for low regularity potential and/or nonlinearity}
In this subsection, we will establish a first-order error bound for the sEWI under low regularity assumptions \cref{eq:A1} made in \cref{thm:first_order_semi}.

Under similar assumptions, a first-order error bound has recently been shown for the first-order EWI in \cite{bao2023_EWI}. Here, following a similar idea, we can prove the same first-order error bound for the sEWI. 

To follow the proof for the one-step method, we first rewrite the sEWI scheme \cref{eq:sEWI_scheme} in matrix form. We define $\phit^\tau:[H^2_\text{per}(\Omega)]^2 \rightarrow H^2_\text{per}(\Omega)$ as
\begin{equation}\label{eq:phit_def}
	\phit^\tau(\phi_1, \phi_0) = e^{2i \tau \Delta} \phi_0 - 2i \tau e^{i\tau\Delta} \vphis(\tau\Delta) B(\phi_1), \quad \phi_0, \phi_1 \in H^2_\text{per}(\Omega). 
\end{equation}
Recalling \cref{eq:sEWI_scheme}, one immediately has 
\begin{equation*}
	\psin{n+1} = \phit^\tau(\psin{n}, \psin{n-1}), \qquad n \geq 1. 
\end{equation*}
Introducing a semi-discrete numerical flow $\Phit^\tau: [H_\text{per}^2(\Omega)]^2 \rightarrow [H_\text{per}^2(\Omega)]^2 $ as
\begin{equation}\label{eq:numerical_flow_semi}
	\Phit^\tau 
	\begin{pmatrix}
		\phi_1 \\
		\phi_0
	\end{pmatrix} = 
	\Aop{\tau}
	\begin{pmatrix}
		\phi_1 \\
		\phi_0
	\end{pmatrix}
	+ \tau \Hop(\phi_1) = 
	\begin{pmatrix}
		\phit^\tau(\phi_1, \phi_0) \\
		\phi_1
	\end{pmatrix}, \quad \phi_0, \phi_1 \in H^2_\text{per}(\Omega), 
\end{equation}
where $\phit^\tau$ is defined in \cref{eq:phit_def} and 
\begin{equation}\label{eq:def_A_H}
	\Aop{t} := \begin{pmatrix}
		0 & e^{2it\Delta}\\
		I & 0
	\end{pmatrix}, \quad
	\Hop(\phi) := \begin{pmatrix}
		- 2i e^{i \tau \Delta} \vphis(\tau\Delta) B(\phi) \\
		0
	\end{pmatrix}, \quad \phi \in H^2_\text{per}(\Omega). 
\end{equation}
Define the semi-discrete solution vector $\Psin{n} (1 \leq n \leq T/\tau) \in [H^2_\text{per}(\Omega)]^2$ as
\begin{equation}
	\Psin{n}:=(\psin{n}, \psin{n-1})^T, \qquad 1 \leq n \leq T/\tau. 
\end{equation}
Then the sEWI scheme \cref{eq:sEWI_scheme} can be equivalently written in matrix form as
\begin{equation}\label{eq:sEWI_matrix}
	\begin{aligned}
		\Psin{n+1} 
		&= \Phit^\tau\left(\Psin{n}\right)
		=  
		\Aop{\tau}
		\begin{pmatrix}
			\psin{n} \\
			\psin{n-1}
		\end{pmatrix}
		+ \tau \Hop(\psin{n}), \quad n \geq 1 \\
		\Psin{1}
		&= 
		\begin{pmatrix}
			\psin{1} \\
			\psin{0}
		\end{pmatrix}
		=	
		\begin{pmatrix}
			e^{i\tau\Delta}\psi_0 - i \tau \vphi_1(i \tau \Delta) B(\psi_0) \\
			\psi_0
		\end{pmatrix}.  
	\end{aligned}
\end{equation}
For $ 1 \leq n \leq T/\tau $, we also define the exact solution vector $\Psitn{n}$ as
\begin{equation}\label{eq:Psitn_def}
	\Psitn{n}: = (\psi(t_n), \psi(t_{n-1}))^T. 
\end{equation}

We first present some estimates of the operator $B$ \cref{eq:B_def} from \cite{bao2023_improved}. 
\begin{lemma}\label{lem:diff_B}
	Under the assumptions $ V \in L^\infty(\Omega) $ and $ \sigma > 0 $, for any $ v, w \in L^\infty(\Omega) $ satisfying $ \| v \|_{L^\infty} \leq M $ and $ \| w \|_{L^\infty} \leq M $, we have
	\begin{equation*}
		\| B(v) - B(w) \|_{L^2} \leq C(\| V \|_{L^\infty}, M) \| v - w \|_{L^2}. 
	\end{equation*}
\end{lemma}
Let $ dB(\cdot)[\cdot] $ be the G\^ateaux derivative defined as (see also \cite{bao2023_semi_smooth})
\begin{equation}\label{eq:Gateaux}
	dB(v)[w]:= \lim_{\varepsilon \rightarrow 0} \frac{B(v + \varepsilon w) - B(v)}{\vep}, 
\end{equation}
where the limit is taken for real $\vep$. 
Introduce a continuous function $G:\C \rightarrow \C$ as 
\begin{equation}\label{eq:G_def}
	G(z) = \left\{
	\begin{aligned}
		& f^\prime(|z|^2) z^2=\beta \sigma |z|^{2\sigma-2} z^2, &z \neq 0, \\
		&0, &z=0,
	\end{aligned}
	\right. \qquad z \in \C.
\end{equation}
Plugging the expression of $B$ \cref{eq:B_def} into \cref{eq:Gateaux}, we obtain
\begin{align}\label{eq:dB_def}
	dB(v)[w]=-i \left[ Vw + f(|v|^2) w + f'(|v|^2)|v|^2w + G(v) \overline{w} \right], 
\end{align}
where $G(v)(x) = G(v(x))$ for $x \in \Omega$. Then we have the following. 
\begin{lemma}\label{lem:dB1}
	Under the assumptions $ V \in L^\infty(\Omega) $ and $ \sigma > 0 $, for any $ v, w \in L^2(\Omega) $ satisfying $ \| v \|_{L^\infty} \leq M $, we have
	\begin{equation*}
		\| dB(v)[w] \|_{L^2} \leq C(\| V \|_{L^\infty}, M) \| w \|_{L^2}. 
	\end{equation*}
\end{lemma}
By \cref{lem:diff_B}, for $\Hop$ in \cref{eq:def_A_H}, using the boundedness of $e^{it\Delta}$ and $\vphis(t \Delta)$, we have
\begin{equation}\label{eq:diff_H}
	\| \Hop(v) - \Hop(w) \|_{L^2} \leq 2\| B(v) - B(w) \|_{L^2} \leq C(\| V \|_{L^\infty}, M) \| v - w \|_{L^2}. 
\end{equation}

Then we shall prove \cref{thm:first_order_semi}. By Proposition 3.7 of \cite{bao2023_EWI} with $n=0$, noting
\begin{equation*}
	\Psitn{1} - \Psin{1} = (\psi(t_1), \psi(t_0))^T - (\psin{1}, \psi_0)^T = (\psi(t_1) - \psin{1}, 0)^T, 
\end{equation*}
we have the following error estimate of the first step. 
\begin{proposition}\label{prop:fist_step_error}
	Under the assumptions \cref{eq:A1}, we have
	\begin{equation*}
		\left \| \Psitn{1} - \Psin{1} \right \|_{H^\alpha} \lesssim \tau^{2 - \alpha/2}, \quad 0 \leq \alpha \leq 2. 
	\end{equation*}
\end{proposition}
By Lemma 3.6 in \cite{bao2023_EWI}, for $ v \in C([0, \tau]; L^2(\Omega)) \cap W^{1, \infty}([0, \tau]; L^2(\Omega)) $ and 
\begin{equation*}
	u_1 = -i\int_0^\tau e^{i(\tau - s)\Delta} v(s) \rmd s,  
\end{equation*}
we have 
\begin{equation}\label{eq:H2_1}
	\| \Delta u_1 \|_{L^2} \leq 2\| v \|_{L^\infty([0, \tau]; L^2)} + \tau \| \partial_t v \|_{L^\infty([0, \tau]; L^2)}. 
\end{equation} 
As an immediate corollary, if $ v \in C([-\tau, \tau]; L^2(\Omega)) \cap W^{1, \infty}([-\tau, \tau]; L^2(\Omega)) $ and 
\begin{equation*}
	u_2 = -i\int_{-\tau}^\tau e^{i(\tau - s)\Delta} v(s) \rmd s, 
\end{equation*}
we have
\begin{equation}\label{eq:H2_2}
	\| \Delta u_2 \|_{L^2} \leq 4\| v \|_{L^\infty([-\tau, \tau]; L^2)} + 2\tau \| \partial_t v \|_{L^\infty([-\tau, \tau]; L^2)}. 
\end{equation}

\begin{proposition}\label{prop:local_error_first}
	Under the assumptions \cref{eq:A1}, we have
	\begin{equation*}
		\| \Psitn{n+1} - \Phit^\tau(\Psitn{n}) \|_{H^\alpha} \lesssim \tau^{2 - \alpha/2}, \quad 0 \leq \alpha \leq 2, \quad 1 \leq n \leq T/\tau-1. 
	\end{equation*}
\end{proposition}
\begin{proof}
	Recalling \cref{eq:Psitn_def,eq:numerical_flow_semi}, we have, for $ 1 \leq n \leq T/\tau-1 $, 
	\begin{equation}\label{eq:LTE_def}
		\begin{aligned}
			\Psitn{n+1} - \Phit^\tau(\Psitn{n}) = 
			\begin{pmatrix}
				\psi(t_{n+1}) - \phit^\tau(\psi(t_{n}), \psi(t_{n-1})) \\
				0
			\end{pmatrix} =:
			\begin{pmatrix}
				\mcalL{n} \\
				0
			\end{pmatrix}.  
		\end{aligned} 
	\end{equation}
	Recalling \cref{eq:phit_def,eq:approximation}, we have
	\begin{align}\label{eq:numeric}
		\phit^\tau(\psi(t_{n}), \psi(t_{n-1})) 
		&= e^{2i\tau\Delta} \psi(t_{n-1}) - 2i \tau e^{i\tau\Delta} \vphis(\tau\Delta) B (\psi(t_{n})) \notag \\
		&= e^{2i\tau\Delta} \psi(t_{n-1}) - i\int_{-\tau}^\tau e^{i(\tau - s)\Delta} B(\psi(t_n)) \rmd s. 
	\end{align}
	From the definition of $\mcalL{n}$ in \cref{eq:LTE_def}, subtracting \cref{eq:numeric} from \cref{eq:exact}, we have 
	\begin{equation}\label{eq:LTE}
		\mcalL{n} = -i\int_{-\tau}^\tau e^{i(\tau - s)\Delta} \left[B(\psi(t_n + s)) - B(\psi(t_n))\right] \rmd s, \quad 1 \leq n \leq T/\tau-1.  
	\end{equation}
	By \cref{lem:diff_B,lem:dB1}, noting that $\partial_t B(\psi(t)) = dB(\psi(t))[\partial_t \psi(t)]$, we have, for $g(s) := B(\psi(t_n + s)) - B(\psi(t_n))$ with $-\tau \leq s \leq \tau$, 
	\begin{equation}\label{eq:est_g}
		\| g \|_{L^\infty([-\tau, \tau]; L^2)} \lesssim \tau, \quad \| \partial_t g \|_{L^\infty([-\tau, \tau]; L^2)} \lesssim 1. 
	\end{equation}		
	From \cref{eq:LTE}, using \cref{eq:H2_2}, noting \cref{eq:est_g}, we obtain, 
	\begin{equation}\label{eq:H2_est_LTE}
		\| \Delta \mcalL{n} \|_{L^2} \lesssim \tau. 
	\end{equation}
	From \cref{eq:LTE}, using the isometry property of $e^{it\Delta}$ and \cref{eq:est_g}, we have 
	\begin{equation}\label{eq:L2_est_LTE}
		\| \mcalL{n} \|_{L^2} \leq \int_{-\tau}^\tau \| g(s) \|_{L^2}  \rmd s \leq 2\tau \| g \|_{L^\infty([-\tau, \tau]; L^2)} \lesssim \tau ^2. 
	\end{equation}
	The conclusion then follows from \cref{eq:H2_est_LTE,eq:L2_est_LTE}, and the standard interpolation inequalities immediately. 
\end{proof}

For the operator $\Aop{t}$ defined in \cref{eq:def_A_H}, we have the following estimate, which shall be compared with the isometry property of the free Schr\"odinger group $e^{i t \Delta}$. 
\begin{lemma}\label{lem:A}
	Let $\mathbf{v} = (v_1, v_2)^T \in [H^2_\text{\rm per}(\Omega)]^2$. For any $t\geq0$, we have
	\begin{equation*}
		\| \Aop{t} \mathbf{v} \|_{H^\alpha} = \| \mathbf v \|_{H^\alpha}, \quad  0 \leq \alpha \leq 2. 
	\end{equation*}
\end{lemma}

\begin{proof}
	Recalling \cref{eq:def_A_H} and using the isometry property of $e^{it\Delta}$, we have
	\begin{equation*}
		\| \Aop{t} \mathbf{v} \|_{H^\alpha}^2 = \| e^{2it\Delta} v_2 \|_{H^\alpha}^2 + \| v_1 \|_{H^\alpha}^2 = \| v_2 \|_{H^\alpha}^2 + \| v_1 \|_{H^\alpha}^2 = \| \mathbf{v} \|_{H^\alpha}^2, 
	\end{equation*}
	which completes the proof. 
\end{proof}

By Parseval's identity, noting that $ |\vphis(\theta)| \leq 1 $ for $\theta \in \R$, we immediately have $\vphis(\tau \Delta)$ is bounded from $H^m_\text{per}(\Omega)$ to $H^m_\text{per}(\Omega)$ for all $ m \in \Z^+ $ and $\tau > 0$. Besides, we have the following analogs of Lemma 3.8 in \cite{bao2023_EWI}. 
\begin{lemma}\label{lem:smoothingofphis}
	Let $ v \in L^2(\Omega) $. For any $0 <\tau < 1$, we have
	\begin{equation*}
		\| \vphi_s(\tau\Delta) v \|_{H^{\alpha}} \lesssim \tau^{-\alpha/2} \| v \|_{L^2}, \quad 0 \leq \alpha \leq 2. 
	\end{equation*}
\end{lemma}


With \cref{lem:A,lem:smoothingofphis}, we can establish the $L^\infty$-conditional stability estimate of $\Phit^\tau$ defined in \cref{eq:numerical_flow_semi}. The proof is similar to that of Proposition 3.10 in \cite{bao2023_EWI} and thus is omitted. 
\begin{proposition}[$L^\infty$-conditional stability estimate]\label{prop:stability}
	Assume that $V \in L^\infty(\Omega)$. Let $\mathbf v = (v_1, v_0)^T, \mathbf w = (w_1, w_0)^T $ such that $ v_j, w_j \in H_\text{\rm per}^2(\Omega) $ for $j = 0, 1$ with $ \| v_1 \|_{L^\infty} \leq M $ and $\| w_1 \|_{L^\infty} \leq M$. For any $\tau > 0$, we have, for $0 \leq \alpha \leq 2$, 
	\begin{equation*}
			\left \| \Phit^\tau(\mathbf{v}) - \Phit^\tau(\mathbf{w}) \right \|_{H^\alpha} \leq \| \mathbf v - \mathbf w \|_{H^\alpha} + C(\| V \|_{L^\infty}, M) \tau^{1-\alpha/2} \| \mathbf v - \mathbf w \|_{L^2}.
	\end{equation*}
\end{proposition}


With \cref{prop:fist_step_error,prop:local_error_first,prop:stability}, the proof of \cref{thm:first_order_semi} can be obtained by following the coupled induction argument used in the proof of (3.2) in Theorem 3.1 of \cite{bao2023_EWI}. We sketch the proof here for the convenience of the reader. 

\begin{proof}[Proof of \cref{thm:first_order_semi}]
	Define the error function $\en{n} := \psi(t_n) - \psin{n}$ for $0 \leq n \leq T/\tau$ and let
	\begin{equation}\label{eq:error_n_def}
		\En{n} := \Psitn{n} - \Psin{n} 
		= (\en{n}, \en{n-1})^T, \quad 1 \leq n \leq T/\tau. 
	\end{equation}
	Recalling \cref{eq:sEWI_matrix}, by the triangle inequality, we get, for $0 \leq \alpha \leq 2$ and $1 \leq n \leq T/\tau-1$,  
	\begin{equation}\label{eq:error_eq_1}
		\| \En{n+1} \|_{H^\alpha} \leq \| \Psitn{n+1} - \Phit^\tau(\Psitn{n}) \|_{H^\alpha} + \| \Phit^\tau(\Psitn{n}) - \Phit^\tau(\Psin{n}) \|_{H^\alpha}. 
	\end{equation}
	We shall show that, when $\tau<\tau_0$ with $\tau_0>0$ sufficiently small, for $1 \leq n \leq T/\tau-1$,  
	\begin{equation}\label{eq:assumption}
		\| \En{n} \|_{L^2} \lesssim \tau, \quad \| \En{n} \|_{H^\frac{7}{4}} \lesssim \tau^\frac{1}{8}.  
	\end{equation}
	We shall use an induction argument to prove it. Note that \cref{eq:assumption} holds for $n=1$ by \cref{prop:fist_step_error}. We assume that \cref{eq:assumption} holds for $1 \leq n \leq m \leq T/\tau-1$. To proceed, choosing $\alpha=0$ and $\alpha=7/4$ in \cref{eq:error_eq_1} and using \cref{prop:local_error_first,prop:stability}, we get
	\begin{align}
		&\| \En{n+1} \|_{L^2} \leq (1+C_0\tau) \| \En{n} \|_{L^2}+ C_1 \tau^2, \label{induction:L2}\\
		&\| \En{n+1} \|_{H^\frac{7}{4}} \leq \| \En{n} \|_{H^\frac{7}{4}} + C_0 \tau^\frac{1}{8} \| \En{n} \|_{L^2} + C_1 \tau^{1+\frac{1}{8}}, \label{induction:Halpha}
	\end{align}
	where, under the given regularity assumptions and the assumption for the induction, by the Sobolev embedding $H^\frac{7}{4} \hookrightarrow L^\infty$, $C_0$ and $C_1$ are uniformly bounded. Applying discrete Gronwall's inequality to \cref{induction:L2}, we have $\| \En{m+1} \|_{L^2} \lesssim \tau$. Summing over $n$ from $1$ to $m$ in \cref{induction:Halpha} yields, by the assumption for the induction, $\| \En{m+1} \|_{H^\frac{7}{4}} \lesssim \tau^\frac{1}{8}$. Then we prove \cref{eq:assumption} for $n=m+1$, and thus for all $1 \leq n \leq T/\tau$ by the induction. The rest of the proof follows immediately. More details can be found in \cite{bao2023_EWI}. 
\end{proof}



\subsection{Proof of Theorem \ref{thm:second_order_semi} for ``good" potential and nonlinearity}
In this subsection, we shall prove an optimal second-order error bound for the sEWI under the assumptions \cref{eq:A2}. 

We first recall some additional estimates for the operator $B$, which will be frequently used in the proof. 
\begin{lemma}[Lemma 4.2 in \cite{bao2023_semi_smooth}]\label{lem:diff_B1}
	Under the assumptions that $ V \in W^{1, 4}(\Omega) $ and $ \sigma \geq 1/2 $, for any $ v, w \in H^2(\Omega) $ such that $ \| v \|_{H^2} \leq M $, $ \| w \|_{H^2} \leq M $, we have
	\begin{equation*}
		\| B(v) - B(w) \|_{H^1} \leq C(\| V \|_{W^{1, 4}}, M) \| v - w \|_{H^1}. 
	\end{equation*}
\end{lemma}


\begin{lemma}[Lemma 4.3 in \cite{bao2023_improved}]\label{lem:diff_dB}
	Under the assumptions $ V \in L^\infty(\Omega) $ and $ \sigma \geq 1/2 $, for any $ v_j, w_j \in L^\infty(\Omega) $ satisfying $ \| v_j \|_{L^\infty} \leq M $ and $ \| w_j \|_{L^\infty} \leq M $ with $ j = 1, 2 $, we have
	\begin{equation*}
		\| dB(v_1)[w_1] - dB(v_2)[w_2] \|_{L^2} \leq C(\| V \|_{L^\infty}, M) \left ( \| v_1 - v_2 \|_{L^2} + \| w_1 - w_2 \|_{L^2} \right ). 
	\end{equation*}
\end{lemma}

\begin{lemma}[Lemma 4.4 in \cite{bao2023_improved}]\label{lem:dB_H2}
	Under the assumptions $ V \in H^2(\Omega) $ and $ \sigma \geq 1 $, for any $ v, w \in H^2(\Omega) $ satisfying $ \| v \|_{H^2} \leq M $, $ \| w \|_{H^2} \leq M $, we have
	\begin{equation*}
		\| dB(v)[w] \|_{H^2} \leq C(\| V \|_{H^2}, M). 
	\end{equation*}
\end{lemma}

Then we shall prove \cref{thm:second_order_semi} under the assumptions \cref{eq:A2}. We start with establishing higher order estimates of the local truncation error. 

Define an even analytic function $\vphic$ as
\begin{equation}\label{eq:vphic_def}
	\vphic(\theta) = \frac{\theta\cos(\theta) - \sin(\theta)}{\theta^3}, \qquad \theta \in \R. 
\end{equation}
Note that $\vphic$ defined in \cref{eq:vphic_def} is bounded on $\R$, and thus the operator $\vphic(\tau \Delta)$ is bounded from $H^m_\text{per}(\Omega)$ to $H^m_\text{per}(\Omega)$ for all $m \in \Z^+$ and $\tau > 0$. Moreover, similar to \cref{lem:smoothingofphis}, we have the following. 
\begin{lemma}\label{lem:vphic_bound}
	Let $\phi \in L^2(\Omega)$. For any $ 0 < \tau < 1$, we have
	\begin{equation*}
		\| \vphic(\tau \Delta) \phi \|_{H^{\alpha}} \lesssim \tau^{-\alpha/2} \| \phi \|_{L^2}, \quad 0 \leq \alpha \leq 2. 
	\end{equation*}
\end{lemma}

Then we can obtain the following estimate of the local truncation error. 
\begin{proposition}\label{prop:local_error_second_semi}
	Under the assumptions \eqref{eq:A2}, we have
	\begin{equation*}
		\| \Psitn{n+1} - \Phit^\tau(\Psitn{n}) \|_{H^\alpha} \lesssim \tau^{3-\alpha/2}, \quad 0 \leq \alpha \leq 2, \quad 1 \leq n \leq T/\tau -1. 
	\end{equation*}
\end{proposition}

\begin{proof}
	Recalling \cref{eq:LTE_def}, it suffices to establish the estimate of $\mcalL{n}$. By \cref{eq:LTE}, 
	\begin{align}\label{eq:rn_def}
		\mcalL{n} 
		&= -i\int_{-\tau}^\tau e^{i(\tau - s)\Delta} \left[B(\psi(t_n + s)) - B(\psi(t_n))\right] \rmd s \notag \\
		&= -i \int_{-\tau}^{\tau} e^{i(\tau - s)\Delta} \big[ B(\psi(t_n + s)) - B(\psi(t_n)) - s dB(\psi(t_n))[\partial_t \psi(t_n)] \big] \rmd s \notag \\
		&\quad -i \int_{-\tau}^{\tau} s e^{i(\tau - s)\Delta} \big(dB(\psi(t_n))[\partial_t \psi(t_n)] \big) \rmd s =: r^n_1 + r^n_2.  
	\end{align}
	For $r^n_1$ in \cref{eq:rn_def}, recalling that $\partial_t B(\psi(t)) = dB(\psi(t))[\partial_t \psi(t)]$, we have 
	\begin{align}\label{eq:integrand}
		&B(\psi(t_n + s)) - B(\psi(t_n)) - s dB(\psi(t_n))[\partial_t \psi(t_n)] \notag\\
		&= \int_0^s \big [ \partial_w [B(\psi(t_n+w)] - dB(\psi(t_n))[\partial_t \psi(t_n)] \big] \rmd w \notag \\
		&= \int_0^s \big [ dB(\psi(t_n+w))[\partial_t \psi(t_n + w)] - dB(\psi(t_n))[\partial_t \psi(t_n)] \big] \rmd w. 
	\end{align}
	From \cref{eq:integrand}, using \cref{lem:diff_dB}, we obtain, for $-\tau \leq s \leq \tau$, 
	\begin{align}\label{eq:est_integrand}
		&\| B(\psi(t_n + s)) - B(\psi(t_n)) - s dB(\psi(t_n))[\partial_t \psi(t_n)]  \|_{L^2} \notag\\
		&\leq \left|\int_0^s \| dB(\psi(t_n+w))[\partial_t \psi(t_n + w)] - dB(\psi(t_n))[\partial_t \psi(t_n)] \|_{L^2} \rmd w \right| \notag \\
		&\lesssim \left|\int_0^s \left( \| \psi(t_n+w) - \psi(t_n) \|_{L^2} + \| \partial_t \psi(t_n + w) - \partial_t \psi(t_n) \|_{L^2} \right) \rmd w \right| \notag \\
		&\lesssim s^2 \left( \| \partial_t \psi \|_{L^\infty([t_{n-1}, t_{n+1}]; L^2)} + \| \partial_{tt} \psi \|_{L^\infty([t_{n-1}, t_{n+1}]; L^2)} \right) \lesssim s^2.  
	\end{align}
	Recalling the definition of $r^{n}_1$ in \cref{eq:rn_def}, using \cref{eq:est_integrand} and the isometry property of $e^{it\Delta}$, we get
	\begin{equation}\label{eq:est_r1}
		\| r_1^{n} \|_{L^2} \leq \int_{-\tau}^{\tau} \| B(\psi(t_n + s)) - B(\psi(t_n)) - s dB(\psi(t_n))[\partial_t \psi(t_n)] \|_{L^2} \rmd s \lesssim \tau^3. 
	\end{equation}
	On the other hand, applying \cref{eq:H2_2} to $r^n_1$ in \cref{eq:rn_def}, using the identity $\partial_t B(\psi(t)) = dB(\psi(t))[\partial_t \psi(t)]$, \cref{eq:est_integrand} and \cref{lem:diff_dB}, we have
	\begin{align}\label{eq:est_r1_H2}
		\| \Delta r_1^{n} \|_{L^2} 
		&\lesssim \sup_{-\tau \leq s \leq \tau}\| B(\psi(t_n + s)) - B(\psi(t_n)) - s dB(\psi(t_n))[\partial_t \psi(t_n)]  \|_{L^2} \notag \\
		&\quad + \tau \sup_{-\tau \leq s \leq \tau} \| \partial_t B(\psi(t_n + s)) - dB(\psi(t_n))[\partial_t \psi(t_n)] \|_{L^2} \notag \\
		&\lesssim \tau^2 + \tau \sup_{-\tau \leq s \leq \tau} \| dB(\psi(t_n + s))[\partial_t \psi(t_n + s)] - dB(\psi(t_n))[\partial_t \psi(t_n)] \|_{L^2} \notag \\
		&\lesssim \tau^2 + \tau \left( \|\psi(t_n + s) - \psi(t_n) \|_{L^2} + \|\partial_t \psi(t_n + s) - \partial_t \psi(t_n) \|_{L^2} \right) \lesssim \tau^2. 
	\end{align}
	Combining \cref{eq:est_r1,eq:est_r1_H2}, and using the interpolation inequality, we get
	\begin{equation}\label{eq:rn_1}
		\| r_1^{n} \|_{H^\alpha} \lesssim \tau^{3-\alpha/2}, \quad 0 \leq \alpha \leq 2. 
	\end{equation}
	Then we shall estimate $r^n_2$ in \cref{eq:rn_def}. Using the identity $e^{it\Delta}  = \cos(t\Delta) + i \sin(t\Delta)$ and the symmetry of the integral domain, we have 
	\begin{align}\label{eq:vphic}
		-i \int_{-\tau}^\tau s e^{-is\Delta} \rmd s 
		&= -i\int_{-\tau}^\tau s \cos(s\Delta) \rmd s - \int_{-\tau}^\tau s \sin(s\Delta) \rmd s \notag \\
		&= -2 \int_0^\tau s\sin(s\Delta) \rmd s	= 2 \tau^3 \Delta \vphic(\tau\Delta), 
	\end{align}
	where $\vphic$ is defined in \cref{eq:vphic_def}. Recalling \cref{eq:rn_def}, using \cref{eq:vphic}, we have
	\begin{align}\label{eq:rn_2}
		r^n_2 
		&= -i \int_{-\tau}^{\tau} s e^{i(\tau - s)\Delta} \big(dB(\psi(t_n))[\partial_t \psi(t_n)] \big) \rmd s \notag \\
		&= -i e^{i\tau\Delta} \int_{-\tau}^{\tau} s e^{- is\Delta} \rmd s \big(dB(\psi(t_n))[\partial_t \psi(t_n)] \big) \notag \\
		&= \tau^3 e^{i\tau\Delta} \Delta \vphic(\tau\Delta) dB(\psi(t_n))[\partial_t \psi(t_n)]. 
	\end{align}
	From \cref{eq:rn_2}, using \cref{lem:dB_H2,lem:vphic_bound}, noting that $dB(\psi(t_n))[\partial_t \psi(t_n)] \in H^2_\text{per}(\Omega)$ and that $\vphic(\tau\Delta)$, $\Delta$ and $e^{i\tau\Delta}$ commute, by the isometry property of $e^{i\tau\Delta} $, we obtain
	\begin{equation}\label{eq:rn_2_est}
		\| r^n_2 \|_{H^\alpha} \lesssim \tau^{-\alpha/2} \tau^3 \| \Delta dB(\psi(t_n))[\partial_t \psi(t_n)] \|_{L^2} \lesssim \tau^{3-\alpha/2}, \quad 0 \leq \alpha \leq 2, 
	\end{equation}
	which, together with \cref{eq:rn_1}, yields from \cref{eq:rn_def} that
	\begin{equation}
		\| \mcalL{n} \|_{H^\alpha} \lesssim \tau^{3-\alpha/2}, \quad 0 \leq \alpha \leq 2, \quad 1 \leq n \leq T/\tau -1, 
	\end{equation}
	which completes the proof. 
\end{proof}

\begin{proposition}[$H^1$-stability estimate]\label{prop:stability_H1_semi}
	Assume that $V \in W^{1, 4}(\Omega) \cap H^1_\text{per}(\Omega) $ and $\sigma \geq 1$. Let $\mathbf v = (v_1, v_0)^T, \mathbf w = (w_1, w_0)^T $ such that $ v_j, w_j \in H_\text{\rm per}^2(\Omega) $ for $j=0, 1$ with $ \| v_1 \|_{H^2} \leq M $ and $\| w_1 \|_{H^2} \leq M$. For $\tau>0$, we have
	\begin{equation*}
		\begin{aligned}
			\left \| \Phit^\tau(\mathbf{v}) - \Phit^\tau(\mathbf{w}) \right \|_{H^1} \leq (1+C(\| V \|_{W^{1, 4}}, M)\tau) \| \mathbf v - \mathbf w \|_{H^1}. 
		\end{aligned}
	\end{equation*}
\end{proposition}

\begin{proof}
	Recalling \cref{eq:numerical_flow_semi}, using \cref{lem:A}, and the boundedness of $e^{it\Delta}$ and $\vphis(\tau\Delta)$, we obtain
	\begin{align}\label{eq:diff_Phit}
		\left \| \Phit^\tau(\mathbf{v}) - \Phit^\tau(\mathbf{w}) \right \|_{H^1} 
		&\leq \| \Aop{\tau} (\mathbf v - \mathbf w) \|_{H^1} + \tau \| \Hop(v_1) - \Hop(w_1) \|_{H^1} \notag \\
		&\leq \| \mathbf v - \mathbf w \|_{H^1} + 2\tau \| B(v_1) - B (w_1) \|_{H^1}. 
	\end{align}
	The conclusion follows from \cref{lem:diff_B1} immediately. 
\end{proof}

\begin{proof}[Proof of Theorem \ref{thm:second_order_semi}]
	Recall that $\en{n} = \psi(t_n) - \psin{n}$ and $\En{n} = (\en{n}, \en{n-1})^T$ in the proof of \cref{thm:first_order_semi}. 
	We start with the estimate of $\En{1} = (\en{1}, 0)^T$. Recalling \cref{eq:approx_neq0,eq:sEWI_scheme}, using the isometry property of $e^{it\Delta}$, \cref{lem:diff_B,lem:diff_B1}, we have
	\begin{equation}\label{eq:est_1_semi}
		\| \En{1} \|_{H^\alpha} = \| \en{1} \|_{H^\alpha} \lesssim \tau^2, \quad \alpha=0, 1. 
	\end{equation}
	From \cref{eq:error_eq_1}, using \cref{prop:stability} with $\alpha =0$, and \cref{prop:local_error_second_semi,prop:stability_H1_semi}, we obtain	\begin{align}\label{eq:error_eq_second_semi_norm}
		\| \En{n+1} \|_{H^\alpha} \leq (1 + C_2 \tau) \| \En{n} \|_{H^\alpha} + C_3 \tau^{3-\alpha/2}, \quad \alpha = 0, 1, 
	\end{align}
	where $C_2$ depends, in particular, on $\| \psin{n} \|_{H^2}$. By the uniform $H^2$-norm bound of $\psin{n}$ established in \cref{thm:first_order_semi} and given regularity assumptions, the constants $C_2$ and $C_3$ in \cref{eq:error_eq_second_semi_norm} are uniformly bounded. From \cref{eq:error_eq_second_semi_norm}, using discrete Gronwall's inequality, noting \cref{eq:est_1_semi}, we can obtain the desired result. 
\end{proof}

\section{Error estimates of the full-discretization \cref{eq:sEWI-FS_scheme}}
In this section, we first extend the error estimates of the semi-discretization \cref{eq:sEWI_scheme} to the full-discretization \cref{eq:sEWI-FS_scheme}. Then we present an improved error estimate which holds only at fully discrete level for certain non-resonant time steps. 

\subsection{Main results}\label{sec:main_2}
Let $\psihn{n} (0 \leq n \leq T/\tau)$ be obtained from the full-discretization sEWI-FS scheme \cref{eq:sEWI-FS_scheme}. We first present the fully discrete counterparts of \cref{thm:second_order_semi,thm:first_order_semi}. 
\begin{theorem}[Optimal $L^2$-norm error bound for ``good" potential and nonlinearity]\label{thm:second_order_full}
	Under the assumptions \cref{eq:A2}, there exits $\tau_0>0, h_0>0$ sufficiently small such that when $0 < \tau < \tau_0$ and $0 < h < h_0$, we have
	\begin{equation*}
		\| \psi(\cdot, t_n) - \psihn{n} \|_{L^2} \lesssim \tau^2 + h^4, \quad \| \psi(\cdot, t_n) - \psihn{n} \|_{H^1} \lesssim \tau^\frac{3}{2} + h^3, \quad 0 \leq  n \leq T/\tau.  
	\end{equation*}
\end{theorem}

\begin{theorem}[Error bound for low regularity potential and/or nonlinearity]\label{thm:first_order_full}
	Under the assumptions \cref{eq:A1}, there exits $\tau_0>0, h_0>0$ sufficiently small such that when $0 < \tau < \tau_0$ and $0 < h < h_0$, we have
	\begin{equation*}
		\begin{aligned}
			&\| \psi(\cdot, t_n) - \psihn{n} \|_{L^2} \lesssim \tau + h^2, \quad \| \psihn{n} \|_{H^2} \lesssim 1, \\
			&\| \psi(\cdot, t_n) - \psihn{n} \|_{H^1} \lesssim \tau^\frac{1}{2} + h, \qquad 0 \leq n \leq T/\tau.
		\end{aligned}
	\end{equation*}
\end{theorem}

Recall that time-splitting methods require stronger regularity on potential and nonlinearity than the sEWI at semi-discrete level as discussed in \cref{rem:semi}. Although recent results in \cite{bao2023_improved} indicate that the regularity requirement can be relaxed at fully discrete level for time-splitting methods, however, a time step size restriction still needs to be imposed, which is not required by the sEWI. Moreover, under the same time step size restriction, the regularity requirement of the sEWI can be further relaxed to cover $\sigma \geq 1/2$. These demonstrate the superiority of the sEWI over time-splitting methods.

\begin{theorem}[Improved optimal $L^2$-norm error bound]\label{thm:second_order_improved}
	Under the assumptions \cref{eq:A3}, there exits $\tau_0>0, h_0>0$ sufficiently small such that when $0 < \tau < \tau_0$, $0 < h < h_0$, and $\tau \leq \frac{h^2}{2\pi}$, we have 
	\begin{equation*}
		\| \psi(\cdot, t_n) - \psihn{n} \|_{L^2} \lesssim \tau^2 + h^4, \quad \| \psi(\cdot, t_n) - \psihn{n} \|_{H^1} \lesssim \tau^\frac{3}{2} + h^3, \quad 0 \leq n \leq T/\tau. 
	\end{equation*}
\end{theorem}

\begin{remark}
	According to the numerical results in Section \ref{sec:num} (see also \cite{bao2023_semi_smooth}), when $V \in H^2_\text{per}(\Omega)$ and $\sigma \geq 1/2$, the regularity requirement of the exact solution in \cref{thm:second_order_improved} (i.e., $\psi \in C([0, T]; H^4_\text{\rm per}(\Omega)) \cap C^1([0, T]; H^2(\Omega)) \cap C^2([0, T]; L^2(\Omega))$) is possible for $H^4$-initial datum. Besides, the time step size restriction $ \tau \leq h^2/(2\pi) $ in \cref{thm:second_order_improved} is natural for the balance of temporal and spatial errors. 
\end{remark}


Similar to what we have done in the previous section, we shall start with the proof of \cref{thm:first_order_full} since the assumption of it is more general, and the uniform $H^2$-bound established in it will be useful to prove \cref{thm:second_order_full,thm:second_order_improved}. 

\subsection{Proof of Theorem \ref{thm:first_order_full} for low regularity potential and/or nonlinearity}
In this subsection, we shall prove \cref{thm:first_order_full} under the assumptions \cref{eq:A1}.  

Again, we first rewrite the sEWI-FS scheme \cref{eq:sEWI-FS_scheme} in matrix form. Similar to \cref{eq:phit_def}, we define $\phiht^\tau:(X_N)^2 \rightarrow X_N$ as 
\begin{equation}\label{eq:phiht_def}
	\phiht^\tau(\phi_1, \phi_0) = e^{2i \tau \Delta} \phi_0 - 2i \tau e^{i\tau\Delta} \vphis(\tau\Delta) P_N B(\phi_1), \quad \phi_0, \phi_1 \in X_N, 
\end{equation}
and one has, from \cref{eq:sEWI-FS}, 
\begin{equation}
	\psihn{n+1} = \phiht^\tau(\psihn{n}, \psihn{n-1}), \quad n \geq 1. 
\end{equation} 
Introducing a fully discrete numerical flow $\Phiht^\tau: (X_N)^2 \rightarrow (X_N)^2 $ as
\begin{equation}\label{eq:numerical_flow_full}
	\Phiht^\tau 
	\begin{pmatrix}
		\phi_1 \\
		\phi_0
	\end{pmatrix} = 
	\Aop{\tau}
	\begin{pmatrix}
		\phi_1 \\
		\phi_0
	\end{pmatrix}
	+ \tau P_N \Hop(\phi_1)
	= \begin{pmatrix}
		\phiht^\tau(\phi_1, \phi_0) \\
		\phi_1
	\end{pmatrix}, \quad \phi_0, \phi_1 \in X_N, 
\end{equation}
where $\Aop{\tau}$ and $\Hop$ are defined in \cref{eq:def_A_H}, and $P_N$ is understood as acting on each component. 
Define the fully discrete solution vector $\Psihn{n} (1 \leq n \leq T/\tau) \in (X_N)^2$ as
\begin{equation}\label{eq:Psihn_def}
	\Psihn{n}:=(\psihn{n}, \psihn{n-1})^T, \qquad 1 \leq n \leq T/\tau. 
\end{equation}
Then \eqref{eq:sEWI-FS} can be equivalently written into the matrix form as 
\begin{equation}\label{eq:sEWI-FS_matrix}
	\begin{aligned}
		\Psihn{n+1} 
		&= \Phiht^\tau\left(\Psihn{n}\right)
		=  
		\Aop{\tau}
		\begin{pmatrix}
			\psihn{n} \\
			\psihn{n-1}
		\end{pmatrix}
		+ \tau P_N\Hop(\psihn{n}), \quad n \geq 1 \\
		\Psihn{1}
		&= 
		\begin{pmatrix}
			\psihn{1} \\
			\psihn{0}
		\end{pmatrix}
		=	
		\begin{pmatrix}
			e^{i \tau \Delta}P_N\psi_0 - i \tau \vphi_1(i \tau \Delta) P_N B(P_N \psi_0) \\
			P_N \psi_0
		\end{pmatrix}.  
	\end{aligned}
\end{equation}
To prove \cref{thm:first_order_full}, we shall first estimate the error between the semi-discrete solution obtained from the sEWI \cref{eq:sEWI_scheme} and the fully discrete solution obtained from the sEWI-FS \cref{eq:sEWI-FS_scheme}. 

\begin{proposition}\label{prop:local_error_first_full}
	Under the assumptions \cref{eq:A1}, when $0<\tau<1$, we have
	\begin{equation*}
		\| P_N \Psin{n} - \Psihn{n} \|_{L^2} \lesssim h^2, \quad 1 \leq n \leq T/\tau. 
	\end{equation*}
\end{proposition}

\begin{proof}
	For the first step, recalling \cref{eq:sEWI_matrix,eq:sEWI-FS_matrix}, by \cref{lem:diff_B}, the boundedness of $\vphis(t\Delta)$ and $P_N$, and the standard projection error estimates of $P_N$, noting that $P_N$ commutes with $e^{it\Delta}$ and $\vphis(\tau\Delta)$, we have  
	\begin{align}\label{eq:first_step_error_full}
		\| P_N \Psin{1} - \Psihn{1} \|_{L^2} 
		&= \tau \| \vphis(\tau\Delta) P_N (B(\psi_0) - B(P_N \psi_0)) \|_{L^2} \leq \tau \| B(\psi_0) - B(P_N \psi_0) \|_{L^2} \notag \\
		&\lesssim \tau \| \psi_0 - P_N \psi_0 \|_{L^2} \lesssim \tau h^2. 
	\end{align}
	Recalling \cref{eq:numerical_flow_semi,eq:numerical_flow_full}, using \cref{eq:diff_H} and the standard projection error estimates, and noting that $\Aop{\tau}$ and $P_N$ commute, we have, for $\mathbf{v} = (v_1, v_0)^T \in [H^2_\text{per}(\Omega)]^2. $
	\begin{align}\label{eq:local_error_full}
		&\| P_N \Phit^\tau(\mathbf{v}) - \Phiht^\tau(P_N \mathbf{v}) \|_{L^2}  \notag \\
		&= \|\Aop{\tau} P_N \mathbf v + \tau P_N \Hop (v_1) - \Aop{\tau} P_N \mathbf v - \tau P_N \Hop(P_N v_1) \|_{L^2} \notag \\
		&= \tau \| P_N \Hop(v_1) - P_N \Hop(P_N v_1) \|_{L^2} \lesssim \tau \| v_1 - P_N v_1 \|_{L^2} \lesssim \tau h^2. 
	\end{align}
	Moreover, for any $\mathbf{v}=(v_1, v_0)^T, \mathbf{w} = (w_1, w_0)^T \in (X_N)^2$ such that $\| v_1 \|_{L^\infty} \leq M$ and $\| w_1 \|_{L^\infty} \leq M$, by using \cref{lem:A} and \cref{eq:diff_H}, and the boundedness of $P_N$, we have 
	\begin{equation}\label{eq:stability_full}
		\| \Phiht^\tau(\mathbf{v}) - \Phiht^\tau(\mathbf{w}) \|_{L^2} \leq (1 + C(\| V \|_{L^\infty}, M) \tau) \| \mathbf{v} - \mathbf{w} \|_{L^2}. 
	\end{equation}
	By the triangle inequality, \cref{eq:local_error_full,eq:stability_full}, noting that $\Psin{n+1} = \Phit^\tau(\Psin{n})$ and $\Psihn{n+1} = \Phiht^\tau(\Psihn{n})$, recalling the uniform $H^2$-bound of $\psin{n}$ in \cref{thm:first_order_semi}, we have
	\begin{align}\label{eq:error_eq_first_order_full}
		&\| P_N \Psin{n+1} - \Psihn{n+1} \|_{L^2} \notag \\
		&\leq \| P_N \Psin{n+1} - \Phiht^\tau(P_N \Psin{n}) \|_{L^2} + \| \Phiht^\tau(P_N \Psin{n}) - \Phiht^\tau(\Psihn{n}) \|_{L^2} \notag \\
		&\leq C_5\tau h^2 + (1+C_4\tau) \| P_N \Psin{n} - \Psihn{n} \|_{L^2}, \quad 1 \leq n \leq T/\tau-1, 
	\end{align}
	where $C_4$ depends, in particular, on $\| \psihn{n} \|_{L^\infty}$. From \cref{eq:error_eq_first_order_full}, using discrete Gronwall's inequality and \cref{eq:first_step_error_full}, along with the induction argument and inverse estimates $\| \phi \|_{L^\infty} \lesssim h^{-1/2} \| \phi \|_{L^2}$ for $\phi \in X_N$ \cite{book_spectral} to control the $L^\infty$-norm of $\Psihn{n}$ (and thus control $C_4$), we can prove the desired result (see Proposition 4.3 in \cite{bao2023_EWI} for details). 
\end{proof}

\begin{proof}[Proof of Theorem \ref{thm:first_order_full}]
	By \cref{prop:local_error_first_full,thm:first_order_semi}, error estimates of $P_N$ and the inverse inequality $\| \phi \|_{H^\alpha} \lesssim h^{-\alpha} \| \phi \|_{L^2}$ for $\phi \in X_N$ \cite{book_spectral}, we get 
	\begin{align}
		\| \psi(\cdot, t_n) - \psihn{n} \|_{H^\alpha} 
		&\leq \| \psi(\cdot, t_n) - \psin{n} \|_{H^\alpha} + \| \psin{n}  - P_N \psin{n} \|_{H^\alpha} + \| P_N \psin{n} - \psihn{n} \|_{H^\alpha} \notag \\
		&\lesssim \tau^{1-\alpha/2} + h^{2-\alpha} + h^{-\alpha} \| P_N \Psin{n} - \Psihn{n} \|_{L^2} \lesssim \tau^{1-\alpha/2} + h^{2-\alpha}, 
	\end{align}
	which proves Theorem \ref{thm:first_order_full} by taking $\alpha =0, 1, 2$. 
\end{proof}

\subsection{Proof of Theorem \ref{thm:second_order_full} for ``good" potential and nonlinearity}
We define the fully discrete counterpart of $\mcalL{n}$ in \cref{eq:LTE_def} as
\begin{equation}\label{eq:En_def}
	\mathcal{E}^n: = P_N \psi(t_{n+1}) - \phiht^\tau(P_N \psi(t_{n}), P_N \psi(t_{n-1})), \quad 1 \leq n \leq T/\tau -1, 
\end{equation}
where $\phiht^\tau$ is defined in \cref{eq:phiht_def}. 
We first separate $\mathcal{E}^n$ into two parts, where the optimal error bound of one part can be obtained under weaker regularity assumptions than the other part. This decomposition will also be useful in the proof of Theorem \ref{thm:second_order_improved} to obtain the improved optimal error bound for non-resonant time steps. 
\begin{proposition}\label{prop:local_error_decomp}
	Under the assumptions that $V \in H^2_\text{\rm per}(\Omega)$, $\sigma \geq 1/2$, and $\psi \in C([0, T]; H^4_\text{\rm per}(\Omega)) \cap C^1([0, T]; H^2(\Omega)) \cap C^2([0, T]; L^2(\Omega))$, we have, for $\mathcal{E}^n$ in \cref{eq:En_def}, 
	\begin{equation}\label{eq:En_decomp}
		\mathcal{E}^n = \mathcal{E}_{\rm dom}^n + \mathcal{E}^n_2, \quad 1 \leq n \leq T/\tau-1, 
	\end{equation}
	where $ \| \mathcal{E}^n_2 \|_{H^\alpha} \lesssim \tau^{3-\alpha/2} + \tau h^{4-\alpha}$ for $0 \leq \alpha \leq 2$ and 
	\begin{equation}\label{eq:En_dom_def}
		\mathcal{E}_{\rm dom}^n = 2\tau^3 \Delta e^{i \tau \Delta} \vphic(\tau \Delta) P_N dB(\psi(t_n))[\partial_t \psi(t_n)], 
	\end{equation}
	with $\vphic:\R \rightarrow \R$ defined in \cref{eq:vphic_def}. 
\end{proposition}

\begin{proof}
	Applying $P_N $ on both sides of \cref{eq:exact}, noting $P_N$ and $e^{it\Delta}$ commute, we get
	\begin{equation}\label{eq:duhamel_exact_truncate}
		P_N \psi(t_{n+1}) = e^{2i\tau\Delta} P_N \psi(t_{n-1}) -i \int_{-\tau}^{\tau} e^{i(\tau - s)\Delta} P_N B(\psi(t_n + s)) \rmd s. 
	\end{equation}
	Recalling \cref{eq:phiht_def,eq:approximation}, similar to \cref{eq:numeric}, we have
	\begin{align}\label{eq:numerical_solution_full}
		\phiht^\tau(P_N \psi(t_{n}), P_N \psi(t_{n-1})) 
		&= e^{2i\tau\Delta} P_N \psi(t_{n-1}) - 2 i \tau \vphi_s(\tau \Delta) P_N B(P_N \psi(t_{n})) \notag \\
		&= e^{2i\tau\Delta} P_N \psi(t_{n-1}) -i \int_{-\tau}^{\tau} e^{i(\tau - s)\Delta} P_N B(P_N \psi(t_n)) \rmd s. 
	\end{align}
	Subtracting \cref{eq:numerical_solution_full} from \cref{eq:duhamel_exact_truncate}, and recalling \cref{eq:En_def}, we obtain
	\begin{align}\label{eq:local_error_full_rep}
		\mathcal{E}^n 
		&= P_N \psi(t_{n+1}) - \phiht^\tau(P_N \psi(t_{n}), P_N \psi(t_{n-1})) \notag \\
		&= -i \int_{-\tau}^{\tau} e^{i(\tau - s)\Delta} \left[ P_N B(\psi(t_n + s)) - P_N B(P_N \psi(t_n)) \right] \rmd s. 
	\end{align}
	From \cref{eq:local_error_full_rep}, we have
	\begin{align}\label{eq:rh_def}
		\mathcal{E}^n 
		&= -i \int_{-\tau}^{\tau} e^{i(\tau - s)\Delta} \big[ P_N B(\psi(t_n + s)) - P_N B(\psi(t_n)) \big] \rmd s \notag \\
		&\quad -i \int_{-\tau}^{\tau} e^{i(\tau - s)\Delta} \big[ P_N B(\psi(t_n)) - P_N B(P_N \psi(t_n)) \big] \rmd s = : r^n + r_h^n. 
	\end{align}
	By \cref{lem:diff_B}, the boundedness of $e^{it\Delta}$ and $P_N$, and the standard projection error estimates of $P_N$,
	\begin{equation}\label{eq:est_rh}
		\| r_h^n \|_{L^2} \leq \int_{-\tau}^{\tau} \| B(\psi(t_n)) - B(P_N \psi(t_n)) \|_{L^2} \rmd s \lesssim \tau \| \psi(t_n) - P_N \psi(t_n) \|_{L^2} \lesssim \tau h^4. 
	\end{equation}
	Note that the definition of $r_h^n$ in \cref{eq:rh_def} implies $r_h^n \in X_N$. Hence, by the inverse inequality 
	and \cref{eq:est_rh}, we obtain
	\begin{equation}
		\| r_h^n \|_{H^\alpha} \lesssim \tau h^{4-\alpha}, \quad 0 \leq \alpha \leq 2. 
	\end{equation}
	For $r^n$ defined in \cref{eq:rh_def}, recalling \cref{eq:rn_def,eq:En_dom_def,eq:rn_2}, we have
	\begin{equation}
		r^n = P_N \mcalL{n} = P_N r^n_1 + P_N r^n_2 = P_N r^n_1 + \mathcal{E}_\text{dom}^n.  
	\end{equation}
	Moreover, one may check that to obtain the estimate of $r^n_1$ in \cref{eq:rn_1}, it suffices to assume that $\sigma \geq 1/2$ instead of $\sigma \geq 1$. Hence, \cref{eq:rn_1} is also valid here. By \cref{eq:rn_1} and the boundedness of $P_N$, noting that $r^n_1 \in H^2_\text{per}(\Omega)$, we have
	\begin{equation}
		\| P_N r^n_1 \|_{H^\alpha} \leq \| r^n_1 \|_{H^\alpha} \lesssim \tau^{3-\alpha/2}, \quad 0\leq \alpha \leq 2. 
	\end{equation}
	The proof can be completed by taking $\mathcal{E}_2^n = P_N r^n_1 + r^n_h$ and noting that $\mathcal{E}^n = P_N r^n_1 + r^n_h + \mathcal{E}_\text{dom}^n$. 
\end{proof}

As a corollary of \cref{prop:local_error_decomp}, we have
\begin{corollary}[Local truncation error]\label{cor:local_error_full_second}
	Under the assumptions \cref{eq:A2}, we have, for $1 \leq n \leq T/\tau-1$, 
	\begin{equation}
		\| P_N \Psitn{n+1} - \Phiht^\tau( P_N \Psitn{n}) \|_{H^\alpha} \lesssim \tau^{3-\alpha/2} + \tau h^{4-\alpha}, \quad 0 \leq \alpha \leq 2.
	\end{equation}
\end{corollary}

\begin{proof}
	Recalling \cref{eq:Psitn_def,eq:numerical_flow_full,eq:En_def}, we have
	\begin{align}
			&P_N \Psitn{n+1} - \Phiht^\tau(P_N \Psitn{n}) \notag \\ 
			&= 
			\begin{pmatrix}
				P_N \psi(t_{n+1}) - \phiht^\tau(P_N \psi(t_{n+1}), P_N\psi(t_n)) \\
				P_N \psi(t_n) - P_N \psi(t_n)
			\end{pmatrix} =
			\begin{pmatrix}
				\mathcal{E}^n \\
				0
			\end{pmatrix}. \label{eq:LTE_full}  
		\end{align}
	It follows that
	\begin{equation}\label{eq:reduce_En}
		\| P_N \Psitn{n+1} - \Phiht^\tau( P_N \Psitn{n}) \|_{H^\alpha} = \| \mathcal{E}^n \|_{H^\alpha}, \quad 0 \leq \alpha \leq 2. 
	\end{equation}
	By \cref{prop:local_error_decomp}, we have
	\begin{equation}\label{eq:L2_est_En}
		\|  \mathcal{E}^n  \|_{H^\alpha} \lesssim \tau^{3-\alpha/2} + \tau h^{4-\alpha} + \| \mathcal{E}_\text{dom}^n \|_{H^\alpha}. 
	\end{equation}
	Recalling that $\mathcal{E}^n_\text{dom} = P_N r^n_2$, using \cref{eq:rn_2_est} and the boundedness of $P_N$, we obtain, 
	\begin{align}
		\| \mathcal{E}_\text{dom}^n \|_{H^\alpha} \lesssim \| r^n_2 \|_{H^\alpha} \lesssim \tau^{3-\alpha/2}, \quad 0 \leq \alpha \leq 2,  
	\end{align}
	which, plugged into \cref{eq:L2_est_En}, completes the proof. 
\end{proof}

Similar to Proposition \ref{prop:stability_H1_semi}, we have the conditional $H^1$-stability estimate. 
\begin{proposition}[$H^1$-stability estimate]\label{prop:stability_H1}
	Assume that $V \in W^{1, 4}(\Omega) \cap H^1_\text{\rm per}(\Omega)$ and $\sigma \geq 1$. Let $\mathbf v = (v_1, v_0)^T, \mathbf w = (w_1, w_0)^T $ such that $ v_j, w_j \in H_\text{\rm per}^2(\Omega) $ for $j=0, 1$ with $ \| v_1 \|_{H^2} \leq M $ and $\| w_1 \|_{H^2} \leq M$. Then we have
	\begin{equation*}
		\begin{aligned}
			\left \| \Phiht^\tau(\mathbf{v}) - \Phiht^\tau(\mathbf{w}) \right \|_{H^1} \leq (1+C(\| V \|_{W^{1, 4}}, M)\tau) \| \mathbf v - \mathbf w \|_{H^1}. 
		\end{aligned}
	\end{equation*}
\end{proposition}

With \cref{cor:local_error_full_second} and \cref{prop:stability_H1}, we can prove \cref{thm:second_order_full} by using the standard argument as in the proof of \cref{thm:second_order_semi}. 
\begin{proof}[Proof of Theorem \ref{thm:second_order_full}]
	Let $\ehn{n} = P_N \psi(t_n) - \psihn{n}$ for $0 \leq n \leq T/\tau$ and let
	\begin{equation}\label{eq:error_mat_def}
		\Ehn{n} := P_N \Psitn{n} - \Psihn{n}
		= (\ehn{n}, \ehn{n-1})^T, \quad 1 \leq n \leq T/\tau. 
	\end{equation}
	By the standard projection error estimates
	\begin{equation}\label{eq:reduce_en}
		\| \psi(t_n) - P_N \psi(t_n) \|_{H^\alpha} \lesssim h^{4-\alpha}, \quad \alpha = 0, 1, \quad 0 \leq n \leq T/\tau, 
	\end{equation}
	we only need to obtain the estimates of $\ehn{n}$. 
	
	We start with $\Ehn{1} = (\ehn{1}, \ehn{0})^T = (\ehn{1}, 0)^T$. Recalling the first equality in \cref{eq:approx_neq0} and applying $P_N$ on both sides, noting that $P_N$ commutes with $e^{it\Delta}$, we have
	\begin{equation}\label{eq:duhamel}
		P_N \psi(t_1) = e^{i\tau\Delta} P_N \psi_0 - i \int_0^\tau e^{i(\tau - s)\Delta} P_N B(\psi(s)) \rmd s. 
	\end{equation}
	Subtracting the second equation in \cref{eq:sEWI-FS} from \cref{eq:duhamel}, we get
	\begin{equation}\label{eq:e1}
		\ehn{1} = P_N \psi(t_1) - \psihn{1} = - i \int_0^\tau e^{i(\tau - s)\Delta} P_N [B(\psi(s)) - B(P_N \psi_0)] \rmd s. 
	\end{equation}	
	From \cref{eq:e1}, using standard projection error estimates, the boundednss of $e^{it\Delta}$ and $P_N$, \cref{lem:diff_B,lem:diff_B1}, we obtain
	\begin{align}\label{eq:est_1}
		&\| \Ehn{1} \|_{H^\alpha} 
		= \| \ehn{1} \|_{H^\alpha} 
		\leq \int_0^\tau \| B(\psi(s)) - B(P_N \psi_0) \|_{H^\alpha} \rmd s \notag \\
		&\leq \int_0^\tau \big(\| B(\psi(s)) - B(\psi_0)\|_{H^\alpha} + \| B(\psi_0) - B(P_N \psi_0) \|_{H^\alpha} \big) \rmd s \notag \\
		&\lesssim \tau^2 + \tau h^{4-\alpha} \leq \tau^{2-\alpha/2} + h^{4-\alpha}, \qquad \alpha = 0 , 1. 
	\end{align}
	Recalling \cref{eq:error_mat_def,eq:sEWI-FS_matrix}, we have, for $1 \leq n \leq T/\tau - 1 $, 
	\begin{align}\label{eq:error_eq_second_full}
		\Ehn{n+1} 
		&= P_N \Psitn{n+1} - 
		\Psihn{n+1}
		= P_N \Psitn{n+1} - 
		\Phiht^\tau \left ( \Psihn{n} \right ) \notag\\
		&= P_N \Psitn{n+1} - 
		\Phiht^\tau \left( P_N \Psitn{n} \right)
		+ \Phiht^\tau \left( P_N \Psitn{n} \right)
		- \Phiht^\tau \left ( \Psihn{n} \right ), 
	\end{align}
	which implies, by using \cref{cor:local_error_full_second}, \cref{eq:stability_full} and \cref{prop:stability_H1}, 
	\begin{align}\label{eq:error_eq_second_full_norm}
		\| \Ehn{n+1} \|_{H^\alpha} 
		&\leq \| P_N \Psitn{n+1} - \Phiht^\tau \left( P_N \Psitn{n} \right) \|_{H^\alpha} + \| \Phiht^\tau \left( P_N \Psitn{n} \right) - \Phiht^\tau \left ( \Psihn{n} \right ) \|_{H^\alpha} \notag \\
		&\leq C_7 \left( \tau^{3-\alpha/2} + \tau h^{4-\alpha} \right) + (1 + C_6 \tau) \| \Ehn{n} \|_{H^\alpha}, \quad \alpha = 0, 1, 
	\end{align}
	where $C_6$ depends, in particular, on $\| \psihn{n} \|_{H^2}$. By the uniform $H^2$ bound of $\psihn{n}$ established in \cref{thm:first_order_full} and the assumptions made in \cref{thm:second_order_full}, the constants $C_6$ and $C_7$ in \cref{eq:error_eq_second_full_norm} are uniformly bounded. From \cref{eq:error_eq_second_full_norm}, using discrete Gronwall's inequality and noting \cref{eq:est_1}, we can obtain the desired result. 
\end{proof}

\subsection{Proof of Theorem \ref{thm:second_order_improved} for improved optimal error bound}
In this subsection, we shall present the proof of \cref{thm:second_order_improved}. The key ingredient is to use the technique of Regularity Compensation Oscillation (RCO) to analyze the error cancellation in the accumulation of local truncation errors (i.e., \cref{eq:Jn_def} below). Note that, under the assumptions \cref{eq:A3},  \cref{prop:local_error_decomp} still holds (while \cref{cor:local_error_full_second} does not). 

\begin{proof}[Proof of \cref{thm:second_order_improved}]
	Recall $\ehn{n} = P_N \psi(t_n) - \psihn{n}$ and $\Ehn{n} = P_N \Psitn{n} - \Psihn{n} = (\ehn{n}, \ehn{n-1})^T $ in the proof of \cref{thm:second_order_full}. Similar to \cref{eq:reduce_en}, we only need to establish the $L^2$- and $H^1$-norm error bounds for $\ehn{n}$. 
	From \cref{eq:error_eq_second_full}, by \cref{eq:LTE_full}, we get
	\begin{align}\label{eq:error_eq}
		\Ehn{n+1} 
		&= P_N \Psitn{n+1} - 
		\Phiht^\tau \left( P_N \Psitn{n} \right)
		+ \Phiht^\tau \left( P_N \Psitn{n} \right)
		- \Phiht^\tau \left ( \Psihn{n} \right ) \notag \\
		&= \A\Ehn{n} + 
		\begin{pmatrix}
			\mathcal{E}^n \\
			0
		\end{pmatrix}
		+ \tau P_N \left (\Hop(P_N \psi(t_n)) - \Hop(\psihn{n}) \right ), \quad 1 \leq n \leq T/\tau - 1, 
	\end{align}
	where we abbreviate $\Aop{\tau}$ as $\A$. 
	Iterating \cref{eq:error_eq} and recalling \cref{eq:En_decomp}, for $1 \leq n \leq T/\tau-1$, we have
	\begin{equation}\label{eq:error_eq_iter}
		\begin{aligned}
			\Ehn{n+1} = \A^n \Ehn{1}  + \sum_{k=1}^{n} \A^{n-k} 
			\begin{pmatrix}
				\mathcal{E}_\text{dom}^{k} \\
				0
			\end{pmatrix}
			+ \sum_{k=1}^{n} \A^{n-k} 				
			\begin{pmatrix}
				\mathcal{E}_2^{k} \\
				0
			\end{pmatrix}
			+ \tau \sum_{k=1}^{n} \A^{n-k} \mathbf Z^{k}, 
		\end{aligned} 
	\end{equation}
	where 
	\begin{equation}\label{eq:Zn_def}
		\mathbf Z^n = P_N \Hop(P_N \psi(t_n)) - P_N \Hop(\psihn{n}), \qquad 1 \leq n \leq T/\tau-1. 
	\end{equation}
	In the following, we shall estimate the four terms on the RHS of \cref{eq:error_eq_iter} in the $L^2$-norm respectively. By \cref{lem:A} and \cref{eq:est_1} with $\alpha=0$, we have 
	\begin{equation}\label{eq:e1_L2}
		\| \A^n \Ehn{1} \|_{L^2} = \| \Ehn{1} \|_{L^2} = \| \ehn{1} \|_{L^2} \lesssim \tau^2 + h^4. 
	\end{equation}
	By \cref{lem:A} and \cref{prop:local_error_decomp} with $\alpha = 0$, we get
	\begin{equation}\label{eq:est_2}
		\left \| \sum_{k=1}^{n} \A^{n-k} 				
		\begin{pmatrix}
			\mathcal{E}_2^{k} \\
			0
		\end{pmatrix} \right \|_{L^2}
		\leq \sum_{k=1}^{n} \| \mathcal{E}_2^{k} \|_{L^2} \lesssim n ( \tau^3 + \tau h^4 ) \lesssim \tau^2+ h^4. 
	\end{equation}
	Recalling \cref{eq:Zn_def}, by \cref{lem:A}, \cref{eq:diff_H} and the boundedness of $\vphis(t \Delta)$ and $P_N$, we have
	\begin{equation}\label{eq:est_3}
		\left \| \tau \sum_{k=1}^{n} \A^{n-k} \mathbf Z^{k} \right \|_{L^2} \leq \tau \sum_{k=1}^{n} \| \mathbf Z^{k} \|_{L^2} \lesssim \tau \sum_{k=1}^{n} \| \ehn{k} \|_{L^2} \leq \tau \sum_{k=1}^{n} \| \Ehn{k} \|_{L^2}, 
	\end{equation}
	where the constant depends on $\| V \|_{L^\infty}$, $\| P_N \psi(t_n) \|_{L^\infty}$ and $\| \psihn{n} \|_{L^\infty}$, and thus is uniformly bounded by the given assumptions and the uniform $H^2$-bound of $\psihn{n}$ in \cref{thm:first_order_full}. 
	
	It remains to estimate the second term on the RHS of \eqref{eq:error_eq_iter}, which is the accumulation of the dominant local truncation error:
	\begin{equation}\label{eq:Jn_def}
		\mathcal{J}^n := \sum_{k=1}^{n} \A^{n-k} 
		\begin{pmatrix}
			\mathcal{E}_\text{dom}^{k} \\
			0
		\end{pmatrix} =
		\sum_{k=0}^{n-1} \A^{k} 
		\begin{pmatrix}
			\mathcal{E}^{n-k}_\text{dom} \\
			0
		\end{pmatrix}, \quad 1 \leq n \leq T/\tau -1. 
	\end{equation}
	Taking the Fourier transform of \cref{eq:Jn_def} componentwise, we obtain
	\begin{equation}
		\widehat{(\mathcal{J}^n)}_l := \sum_{k=0}^{n-1} (\A_l)^{k} 
		\begin{pmatrix}
			\widehat{(\mathcal{E}^{n-k}_\text{dom})}_l \\
			0
		\end{pmatrix}, \quad l \in \mathcal{T}_N, \quad 1 \leq n \leq T/\tau -1, 
	\end{equation}
	where $\mathbf{A}_l$ is a unitary matrix defined as
	\begin{equation}\label{eq:Al_def}
		\A_l := \begin{pmatrix}
			0 & e^{-2i\tau\mu_l^2}\\
			1 & 0
		\end{pmatrix}, \quad l \in \mathcal{T}_N. 
	\end{equation}
	Note that $\mathbf{A}_l$ can be decomposed as
	\begin{equation}
		\A_l = \mathbf P_l \mathbf \Lambda_l \mathbf P_l^{\ast}, \quad l \in \mathcal{T}_N,  
	\end{equation}
	where $\mathbf{P}_l^\ast$ is the complex conjugate transpose of $\mathbf{P}_l$ and 
	\begin{equation}\label{eq:eigen_decomp}
		\mathbf P_l = \frac{1}{\sqrt{2}} \begin{pmatrix}
			\lambda_{+, l} & \lambda_{-, l} \\
			1 & 1
		\end{pmatrix}, \quad 
		\mathbf \Lambda_l = \begin{pmatrix}
			\lambda_{+, l} & 0 \\
			0 & \lambda_{-, l}
		\end{pmatrix}, \quad \lambda_{\pm, l} = \pm e^{- i \tau \mu_l^2}. 
	\end{equation}
	Define
	\begin{equation}\label{eq:S_def}
		\mathbf S_{l, n} := \sum_{k=0}^n (\A_l)^k = \mathbf P_l 
		\begin{pmatrix}
			s_{l, n}^+ & 0 \\
			0 & s_{l, n}^-
		\end{pmatrix} \mathbf P_l^{\ast}, \qquad n \geq 0,  
	\end{equation}
	where
	\begin{equation}\label{eq:sum_eigenvalue}
		s_{l, n}^+ = \sum_{k=0}^n \lambda_{+, l}^k, \quad s_{l, n}^- = \sum_{k=0}^n \lambda_{-, l}^k. 
	\end{equation}
	From \cref{eq:sum_eigenvalue}, recalling \cref{eq:eigen_decomp}, noting that $\tau \leq h^2/(2\pi)$ implies $\tau \mu_l^2 \leq \pi/2$ for all $l \in \mathcal{T}_N$, we have 
	\begin{align}
		&\left|s_{l, n}^+\right| = \left | \frac{1 - \lambda_{+, l}^{n+1}}{1-\lambda_{+, l}} \right | \leq \frac{2}{|1 - e^{-i \tau \mu_l^2}|} = \frac{1}{|\sin(\tau \mu_l^2/2)|} \lesssim \frac{1}{\tau \mu_l^2}, \quad 0 \neq l \in \mathcal{T}_N, \label{eq:sp_est} \\
		&\left|s_{l, n}^-\right| = \left | \frac{1 - \lambda_{-, l}^{n+1}}{1-\lambda_{-, l}} \right | \leq \frac{2}{|1+e^{-i\tau\mu_l^2}|} = \frac{1}{|\cos(\tau \mu_l^2 / 2)|} \lesssim 1, \quad l \in \mathcal{T}_N. \label{eq:sm_est}
	\end{align}
	Using the summation by parts formula, we have, for $ l \in \mathcal{T}_N $ and $ 1 \leq n \leq T/\tau -1 $, 
	\begin{align}\label{eq:J_n_sum_by_part}
		\widehat{(\mathcal{J}^n)}_l 
		&= \sum_{k=0}^{n-1} (\mathbf A_l)^{k} 
		\begin{pmatrix}
			\widehat{(\mathcal{E}^{n-k}_\text{dom})}_l \\
			0
		\end{pmatrix} \notag \\
		&= \mathbf S_{l, n-1} \begin{pmatrix}
			\widehat{(\mathcal{E}^{1}_\text{dom})}_l \\
			0
		\end{pmatrix} - \sum_{k=0}^{n-2} \mathbf S_{l, k} 
		\begin{pmatrix}
			\widehat{(\mathcal{E}^{n-k-1}_\text{dom})}_l  - 	\widehat{(\mathcal{E}^{n-k}_\text{dom})}_l \\
			0
		\end{pmatrix}. 
	\end{align}
	Recalling the definition of $\mathcal{E}^n_\text{dom}$ in \cref{eq:En_dom_def}, setting $\phi^n_l := P_N dB(\psi(t_n))[\partial_t \psi(t_n)]$, we have
	\begin{equation}\label{eq:E_dom_hat}
		\widehat{(\mathcal{E}^n_\text{dom})}_l = -2\tau^3 \mu_l^2 e^{-i \tau \mu_l^2} \vphic(\tau \mu_l^2) \widehat{\phi^n_l}, \quad l \in \mathcal{T}_N, \quad 1 \leq n \leq T/\tau - 1, 
	\end{equation}
	which implies, for $1 \leq n \leq T/\tau-1$, 
	\begin{align}
		&\left |\widehat{(\mathcal{E}^n_\text{dom})}_l \right | \lesssim \tau^3 \sup_{x \in \R} |\vphic(x)| \mu_l^2 \left |\widehat{\phi^n_l} \right | \lesssim \tau^3 \mu_l^2 \left |\widehat{\phi^n_l} \right |, \quad l \in \mathcal{T}_N, \label{eq:tau_cube}\\
		&\left |\widehat{(\mathcal{E}^n_\text{dom})}_l \right | \lesssim \tau^2 \sup_{x \in \R} |x\vphic(x)| \left | \widehat{\phi^n_l} \right | \lesssim \tau^2 \left | \widehat{\phi^n_l} \right |, \qquad l \in \mathcal{T}_N. \label{eq:tau_square}
	\end{align}
	From \cref{eq:J_n_sum_by_part}, recalling \cref{eq:S_def,eq:E_dom_hat}, we have, for $ l \in \mathcal{T}_N $, 
	\begin{align}\label{eq:Jn}
		\widehat{(\mathcal{J}^n)}_l 
		&= \frac{1}{\sqrt{2}} \mathbf{P}_l
		\begin{pmatrix}
			s_{l, n-1}^+ \overline{\lambda_{+, l}} \\
			s_{l, n-1}^- \overline{\lambda_{-, l}}
		\end{pmatrix} \widehat{(\mathcal{E}^1_\text{dom})}_l
		- \frac{1}{\sqrt{2}} \mathbf{P}_l \sum_{k=0}^{n-2}
		\begin{pmatrix}
			s_{l, k}^+ \overline{\lambda_{+, l}} \\
			s_{l, k}^- \overline{\lambda_{-, l}}
		\end{pmatrix} 
		\left[\widehat{(\mathcal{E}^{n-k-1}_\text{dom})}_l  - 	\widehat{(\mathcal{E}^{n-k}_\text{dom})}_l \right].
	\end{align}
	From \cref{eq:Jn}, noting that $\| \mathbf{P}_l \|_{2} = 1$ and $ |\lambda_{\pm, l}| = 1$, we have
	\begin{align}\label{eq:Jn_decomp}
		\left |\widehat{(\mathcal{J}^n)}_l \right | 
		\lesssim W^1_l + W^2_l + \sum_{k=0}^{n-2} \left ( W^3_{k, l} + W^4_{k, l} \right ), \quad l \in \mathcal{T}_N, 
	\end{align}
	where
	\begin{equation}\label{eq:W_def}
		\begin{aligned}
			&W^1_l = \left| \widehat{(\mathcal{E}^1_\text{dom})}_l \right|  \left|s_{l, n-1}^+ \right|, \qquad W^2_l = \left| \widehat{(\mathcal{E}^1_\text{dom})}_l \right| \left|s_{l, n-1}^- \right|,\\ 
			&\begin{aligned}
				&W^3_{k, l} =  \left|\widehat{(\mathcal{E}^{n-k}_\text{dom})}_l  - 	\widehat{(\mathcal{E}^{n-k-1}_\text{dom})}_l \right| \left|s_{l, k}^+ \right|, \\
				&W^4_{k, l} =  \left|\widehat{(\mathcal{E}^{n-k}_\text{dom})}_l  - 	\widehat{(\mathcal{E}^{n-k-1}_\text{dom})}_l \right| \left| s_{l, k}^- \right|, 
			\end{aligned}
			\quad k \in \left\{0, 1, \cdots, \frac{T}{\tau} -3 \right\},  
		\end{aligned}
		\qquad l \in \mathcal{T}_N.  
	\end{equation}
	For $ W^1_l $, using \cref{eq:tau_cube,eq:sp_est}, when $\tau \leq h^2/(2\pi)$, we have
	\begin{equation}\label{eq:W1_1}
		| W^1_l | \lesssim \frac{\tau^3 \mu_l^2 }{\tau \mu_l^2} \left |\widehat{\phi^1_l} \right | \lesssim \tau^2 \left |\widehat{\phi^1_l} \right |, \quad 0 \neq l \in \mathcal{T}_N, 
	\end{equation}
	from which, noting that $|W_l^1| = 0$ when $l=0$, we obtain
	\begin{equation}\label{eq:W1}
		| W^1_l | \lesssim \tau^2 \left |\widehat{\phi^1_l} \right |, \quad l \in \mathcal{T}_N.  
	\end{equation} 
	For $ W^2_l $, using \cref{eq:tau_square,eq:sm_est}, when $\tau \leq h^2/(2\pi)$, we have
	\begin{equation}\label{eq:W2}
		| W^2_l | \lesssim \tau^2 \left |\widehat{\phi^1_l} \right |, \quad l \in \mathcal{T}_N.  
	\end{equation}
	Similarly, for $ W^3_{k, l} $ and $ W^4_{k, l} $, when $\tau \leq h^2/(2\pi)$, we have
	\begin{equation}\label{eq:W34}
		| W^3_{k, l} | \lesssim \tau^2 \left |\widehat{\phi^{n-k}_l} - \widehat{\phi^{n-k-1}_l} \right |, \quad | W^4_{k, l} | \lesssim \tau^2 \left |\widehat{\phi^{n-k}_l} - \widehat{\phi^{n-k-1}_l} \right |, \quad l \in \mathcal{T}_N. 
	\end{equation}
	From \cref{eq:Jn_decomp}, using Cauchy inequality, and noting \cref{eq:W1,eq:W2,eq:W34}, we get, for $\tau \leq h^2/(2\pi)$, 
	\begin{align}\label{eq:Jn_hat_est}
		\left |\widehat{(\mathcal{J}^n)}_l \right |^2 
		&\lesssim \left| W_l^1 \right|^2 + \left| W_l^2 \right|^2 + n \sum_{k=0}^{n-2}\left( \left| W_{k, l}^3 \right|^2+\left| W_{k, l}^4 \right|^2 \right) \notag \\
		&\lesssim \tau^4 \left |\widehat{\phi^1_l} \right |^2 + n\tau^4 \sum_{k=1}^{n-1} \left |\widehat{\phi^{k+1}_l} - \widehat{\phi^{k}_l} \right |^2, \quad l \in \mathcal{T}_N, \quad 1 \leq n \leq T/\tau-1. 
	\end{align}
	Using Parseval's identity and \cref{eq:Jn_hat_est}, and changing the order of summation, we have
	\begin{align}\label{eq:est_Jn}
		\| \mathcal{J}^n \|_{L^2}^2 
		&= (b-a) \sum_{l \in \mathcal{T}_N} \left |\widehat{(\mathcal{J}^n)}_l \right |^2 \lesssim \tau^4 \sum_{l \in \mathcal{T}_N} \left |\widehat{\phi^1_l} \right |^2 + n \tau^4 \sum_{l \in \mathcal{T}_N} \sum_{k=1}^{n-1} \left |\widehat{\phi^{k+1}_l} - \widehat{\phi^{k}_l} \right |^2 \notag \\
		&\lesssim \tau^4 \| \phi^1 \|_{L^2}^2 + n \tau^4 \sum_{k=1}^{n-1} \sum_{l \in \mathcal{T}_N} \left |\widehat{\phi^{k+1}_l} - \widehat{\phi^{k}_l} \right |^2 \notag \\
		&\lesssim \tau^4 \| \phi^1 \|_{L^2}^2 + n \tau^4 \sum_{k=1}^{n-1} \| \phi^{k+1} - \phi^k \|_{L^2}^2, \qquad 1 \leq n \leq T/\tau-1. 
	\end{align}
	Recalling that $\phi^n_l = P_N dB(\psi(t_n))[\partial_t \psi(t_n)]$, by \cref{lem:diff_dB} and the boundedness of $P_N$, we have
	\begin{equation*}
		\| \phi^1 \|_{L^2} = \| P_N dB(\psi(t_1))[\partial_t \psi(t_1)] \|_{L^2} \leq \| dB(\psi(t_1))[\partial_t \psi(t_1)] \|_{L^2} \lesssim 1, 
	\end{equation*}
	and
	\begin{align*}
		\| \phi^{k+1} - \phi^k \|_{L^2} 
		&= \left \| P_N dB(\psi(t_{k+1}))[\partial_t \psi(t_{k+1})] - P_N dB(\psi(t_{k}))[\partial_t \psi(t_{k})] \right \|_{L^2} \\
		&\leq \| dB(\psi(t_{k+1}))[\partial_t \psi(t_{k+1})] - dB(\psi(t_{k}))[\partial_t \psi(t_{k})] \|_{L^2} \\
		&\lesssim \| \psi(t_{k+1}) - \psi(t_k) \|_{L^2} + \| \partial_t \psi(t_{k+1}) - \partial_t \psi(t_k) \|_{L^2} \lesssim \tau, 
	\end{align*}
	which, plugged into \cref{eq:est_Jn}, yields
	\begin{equation}\label{eq:est_4}
		\| \mathcal{J}^n \|_{L^2}^2 \lesssim \tau^4 + n^2 \tau^6 \lesssim \tau^4, \qquad 1 \leq n \leq T/\tau-1. 
	\end{equation}
	From \cref{eq:error_eq_iter}, using \cref{eq:e1_L2,eq:est_2,eq:est_3,eq:est_4}, and recalling \cref{eq:Jn_def}, we obtain
	\begin{equation}\label{eq:final_error_eq}
		\| \Ehn{n+1} \|_{L^2} \lesssim \tau^2 + h^4 + \tau \sum_{k=1}^n \| \Ehn{k} \|_{L^2}, \qquad 0 \leq n \leq T/\tau-1, 
	\end{equation} 
	where the case $n = 0$ follows from $\| \Ehn{1} \|_{L^2} \lesssim \tau^2 + h^4$ established in \cref{eq:e1_L2}. Then the proof of the $L^2$-norm error bound is completed by applying discrete Gronwall's inequality to \cref{eq:final_error_eq}. The $H^1$-norm error bound can be obtained by directly applying the inverse inequality $\| \phi \|_{H^1} \lesssim h^{-1} \| \phi \|_{L^2}$ for $\phi \in X_N$ \cite{book_spectral} as
	\begin{equation}\label{eq:improv_inverse}
		\| \ehn{n} \|_{H^1} \lesssim h^{-1}\| \ehn{n} \|_{L^2} \lesssim \tau^2 h^{-1} + h^3 \lesssim \tau^\frac{3}{2} + h^3, \quad 0 \leq n \leq T/\tau, 
	\end{equation}
	where we also use the fact that $\tau \leq h^2/(2\pi)$. The proof is then completed.  
\end{proof}

\section{Numerical results}\label{sec:num}
In this section, we shall show some numerical results to validate our error estimates and to demonstrate the superiority of the sEWI. We first test the convergence of the sEWI under various potential, nonlinearity and initial data. Then we perform some long-time simulation to show the near-conservation of mass and energy. To quantify the error, we introduce the following error functions:
\begin{equation*}
	e_{L^2}(t_n) := \| \psi(\cdot, t_n) - \psihn{n}  \|_{L^2}, \quad  e_{H^1}(t_n) := \| \psi(\cdot, t_n) - \psihn{n} \|_{H^1}, \quad 0 \leq n \leq T/\tau. 
\end{equation*}

We start with ``good" potential and nonlinearity and consider a 1D example. In the NLSE \cref{NLSE}, we choose $\Omega = (-16, 16)$, $\beta = 1$, $\sigma = 1.1$ and
\begin{equation}\label{eq:good}
	V(x) = \left( \frac{x^2 - 4}{16} \right)^{1.51} \times \left(1 - \frac{x^2}{16^2} \right)^2, \quad \psi_0(x) = x|x|^{2.51} e^{-\frac{x^2}{2}}, \quad x \in \Omega. 
\end{equation} 
Here in \cref{eq:good}, the potential function $V$ satisfies $V \in H^2_\text{per}(\Omega)$ and the initial datum $\psi_0 \in H^4(\Omega)$ is intentionally chosen as an odd function. 
Note that, with an odd initial datum, the exact solution will remain an odd function for all time $t \geq 0$, and thus indeed suffer from the low regularity of nonlinearity at the origin. 

In the numerical simulation, we choose $T = 1$. The ``exact" solution is computed by the sEWI-FS method \cref{eq:sEWI-FS_scheme} with $\tau = 10^{-5}$ and $h = 2^{-9}$. The numerical results are presented in \cref{fig:conv_dt_h_H4_ini_good}, where we show temporal errors with respect to the time step size $\tau$ in (a) and spatial errors with respect to the mesh size $h$ in (b). When computing temporal errors, we vary $\tau$ with fixed $h = 2^{-9}$, and when computing spatial errors, we vary $h$ with fixed $\tau = 10^{-5}$. 

From \cref{fig:conv_dt_h_H4_ini_good}, we can observe that for $H^2$-potential, $\sigma=1.1$ and $H^4$-initial data, the temporal convergence is second-order in $ L^2 $-norm and 1.5-order in $ H^1 $-norm; and the spatial convergence is forth-order in $ L^2 $-norm and third-order in $ H^1 $-norm. The results confirm our error estimates in \cref{thm:second_order_semi,thm:second_order_full} and show that they are sharp.  

\begin{figure}[htbp]
	\centering
	{\includegraphics[width=0.475\textwidth]{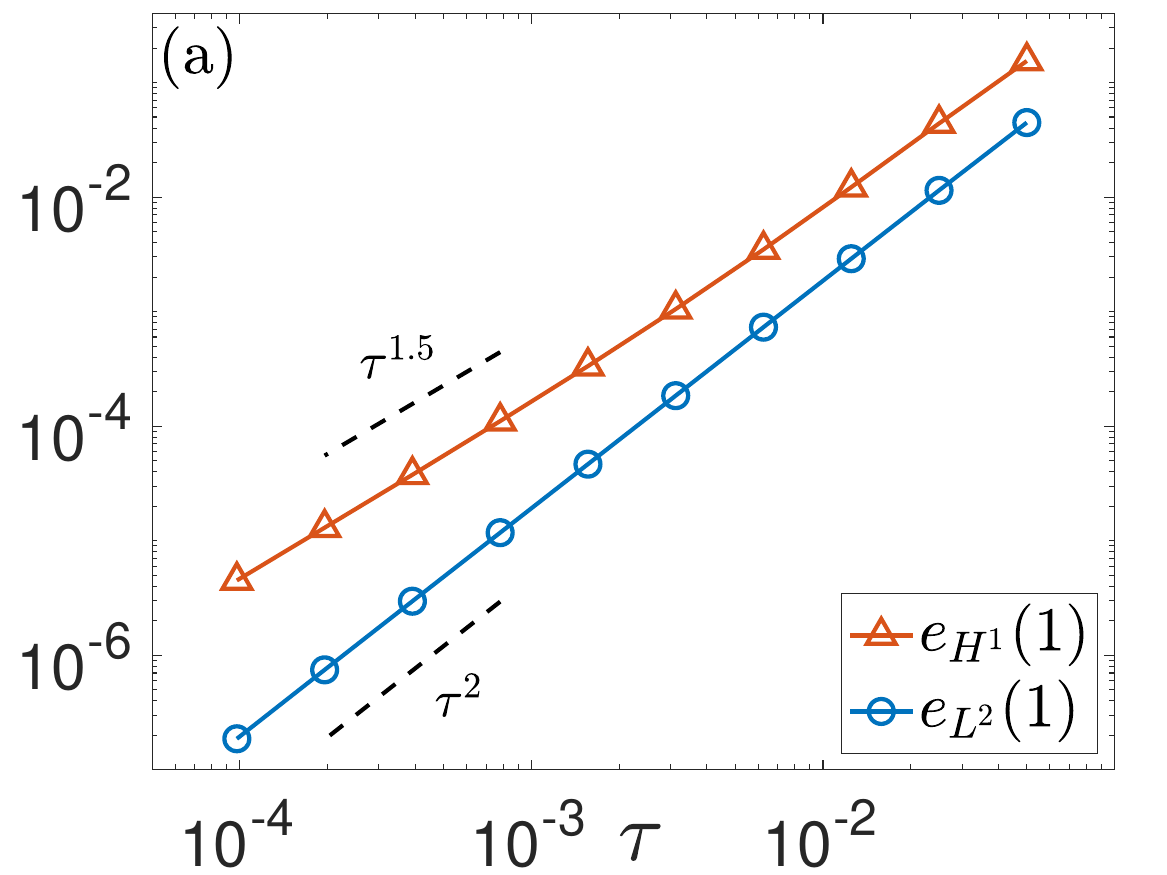}}\hspace{1em}
	{\includegraphics[width=0.475\textwidth]{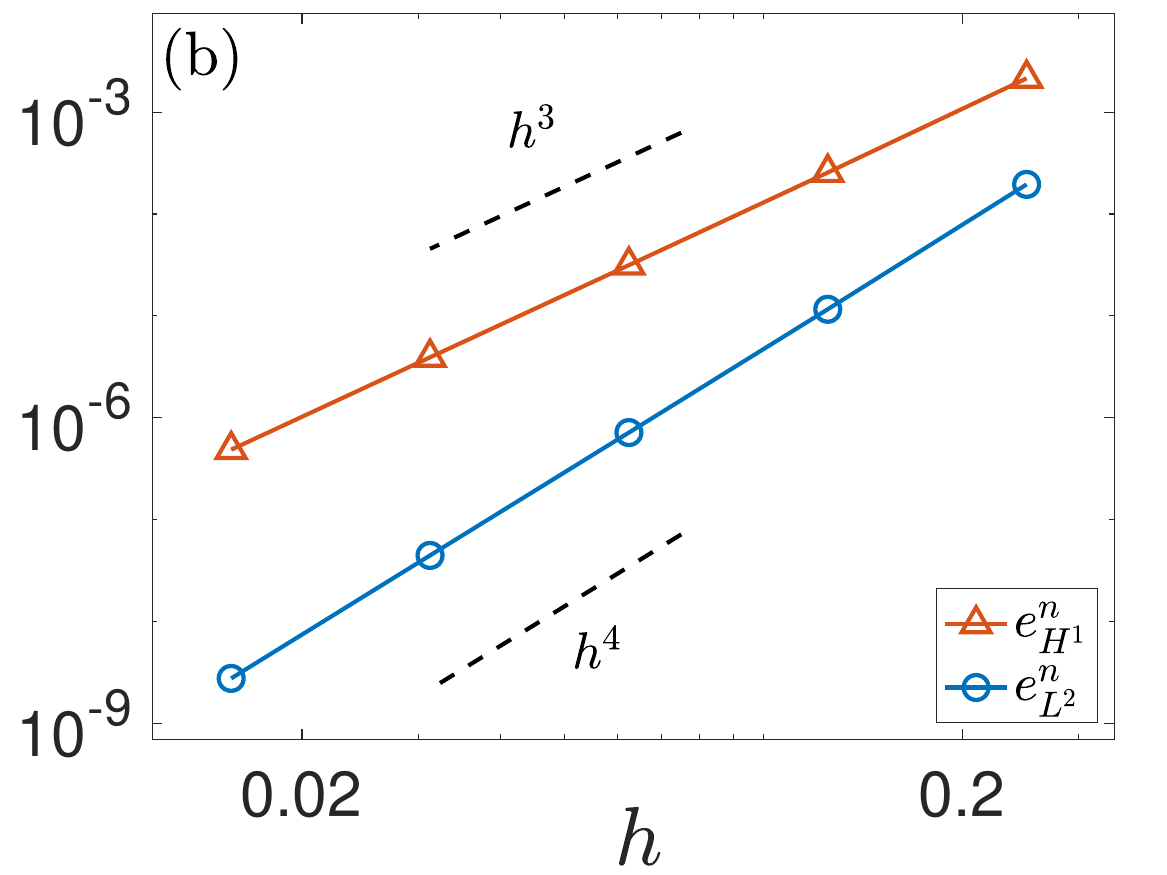}}
	\caption{$L^2$- and $H^1$-errors of the sEWI for the NLSE with $\sigma = 1.1$ and $V \in H^2$ given in \cref{eq:good}: (a) temporal errors and (b) spatial errors.}
	\label{fig:conv_dt_h_H4_ini_good}
\end{figure}

We then consider low regularity potential and nonlinearity. To show the validity of our results in general dimensions, we shall present a 2D example. In \cref{NLSE}, we choose $\Omega = (-8, 8) \times (-8, 8)$, $\beta = 1$, $\sigma = 0.1$ and  
\begin{equation}\label{eq:lowreg}
	V(\vx) = \left\{
	\begin{aligned}
		&10, &&|x| \leq 2, |y| \leq 2, \\
		&0, &&\text{otherwise}, 
	\end{aligned}
	\right.
	, \quad \psi_0(\vx) = x|x|^{0.51} e^{-\frac{|\vx|^2}{2}}, \quad \vx = (x, y)^T \in \Omega. 
\end{equation}
It follows immediately that $V \in L^\infty(\Omega)$ and $\psi_0 \in H^2(\Omega)$ in \cref{eq:lowreg}. 

In computation, we choose $T = 0.25$. The ``exact" solution is computed by the sEWI-FS method with $\tau = 10^{-5}$ and $h_x = h_y = 2^{-7}$, where $h_x$ and $h_y$ are the mesh sizes in $x$ and $y$ directions, respectively. We plot temporal and spatial errors in \cref{fig:conv_dt_h_H2_ini_low_reg}, following the same manner in the previous example. 

From \cref{fig:conv_dt_h_H2_ini_low_reg}, we observe that for $L^\infty$-potential, $\sigma = 0.1$ and $H^2$-initial data, the temporal convergence of the sEWI is first-order in $ L^2 $-norm and half-order in $ H^1 $-norm; and the spatial convergence is second-order in $ L^2 $-norm and first-order in $ H^1 $-norm. The observation validates our error estimates in \cref{thm:first_order_semi,thm:first_order_full} and indicates that they are sharp. 

\begin{figure}[htbp]
	\centering
	{\includegraphics[width=0.475\textwidth]{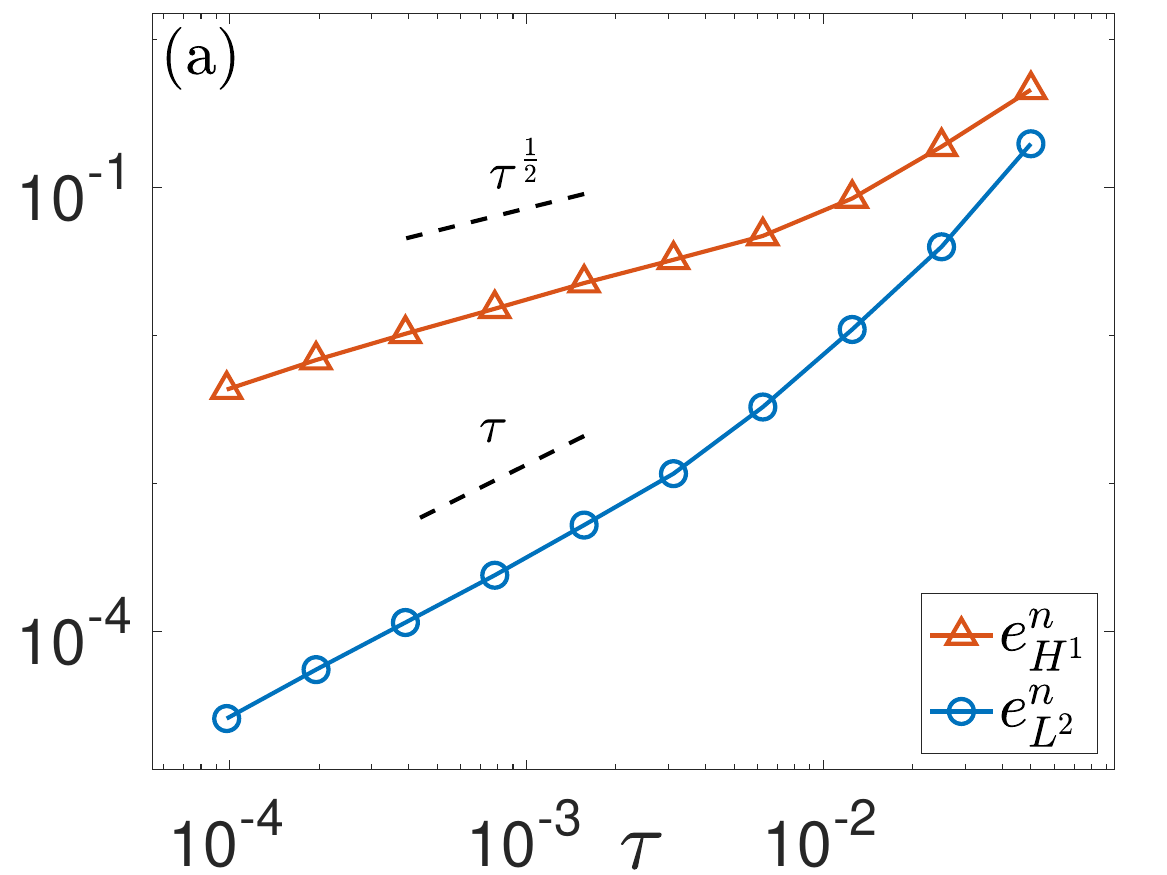}}\hspace{1em}
	{\includegraphics[width=0.475\textwidth]{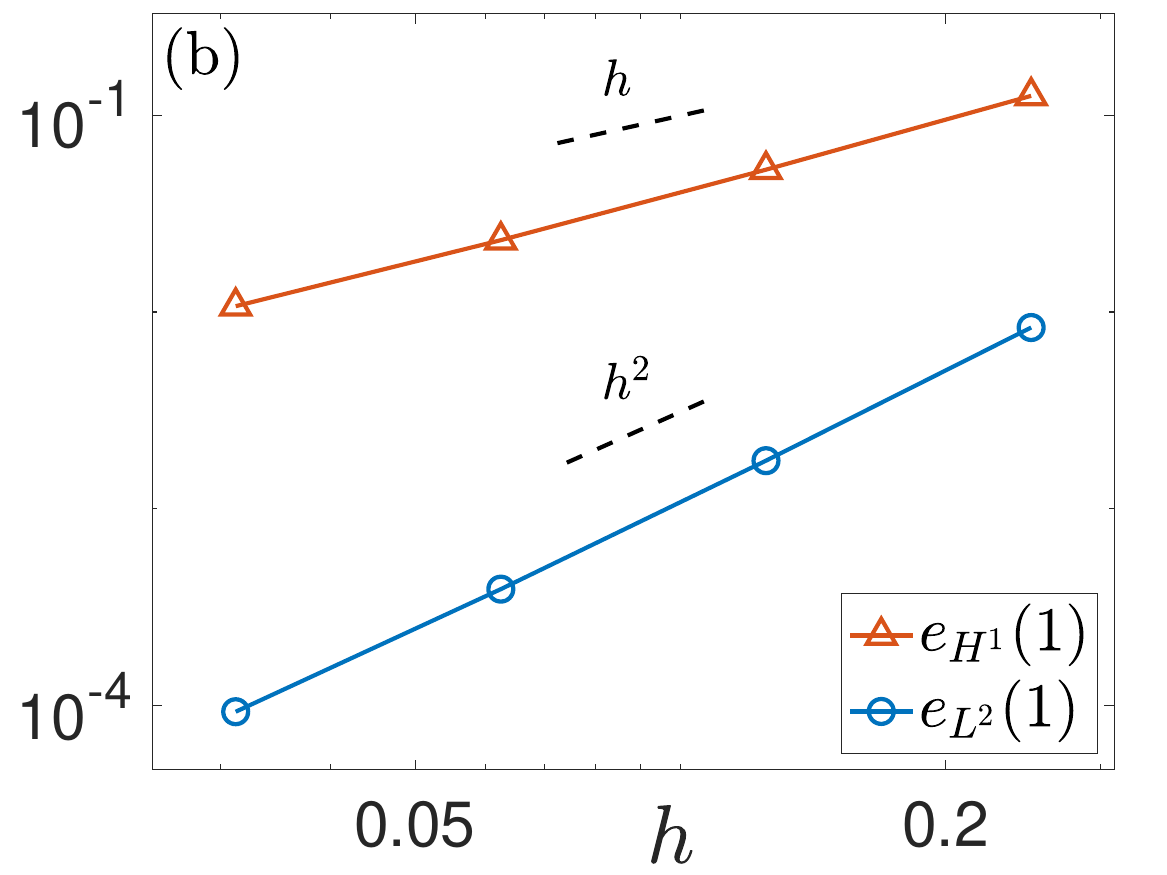}}
	\caption{$L^2$- and $H^1$-errors of the sEWI for the NLSE with $\sigma = 0.1$ and $V \in L^\infty$ given in \cref{eq:lowreg}: (a) temporal errors and (b) spatial errors.}
	\label{fig:conv_dt_h_H2_ini_low_reg}
\end{figure}

Then we consider another example in 1D under improved regularity assumptions. We set $T=1$ and choose $\Omega = (-16, 16)$, $\beta = -1$, $\sigma = 0.5$ and $V(x) \equiv 0 $ with the same $H^4$-initial datum $\psi_0$ given in \cref{eq:good}. The ``exact" solution is computed with $\tau = 10^{-5}$ and $h = 2^{-9}$. Different from the two cases above, we compute the errors of the sEWI-FS with $h = \sqrt{10 \tau}$ to satisfy the constraint $\tau \leq h^2/(2\pi)$. For comparison, we also plot the errors computed by varying $\tau$ with $h = 2^{-9}$ fixed (corresponding to the time semi-discrete case). The numerical results for the two cases are exhibited in \cref{fig:conv_dt_h_H4_ini} (a) and (b), respectively. 

From \cref{fig:conv_dt_h_H4_ini}, we can observe that, the $L^2$-error of the sEWI-FS method is of $O(\tau^2)$ and the $H^1$-error is of $O(\tau^\frac{3}{2})$ with and without the time step size restriction $\tau \leq h^2/(2\pi)$. The numerical results validate our error estimates in \cref{thm:second_order_improved}, but also suggest that the time step size restriction might be relaxed. 
We will try to investigate this phenomenon in our future work. 

\begin{figure}[htbp]
	\centering
	{\includegraphics[width=0.475\textwidth]{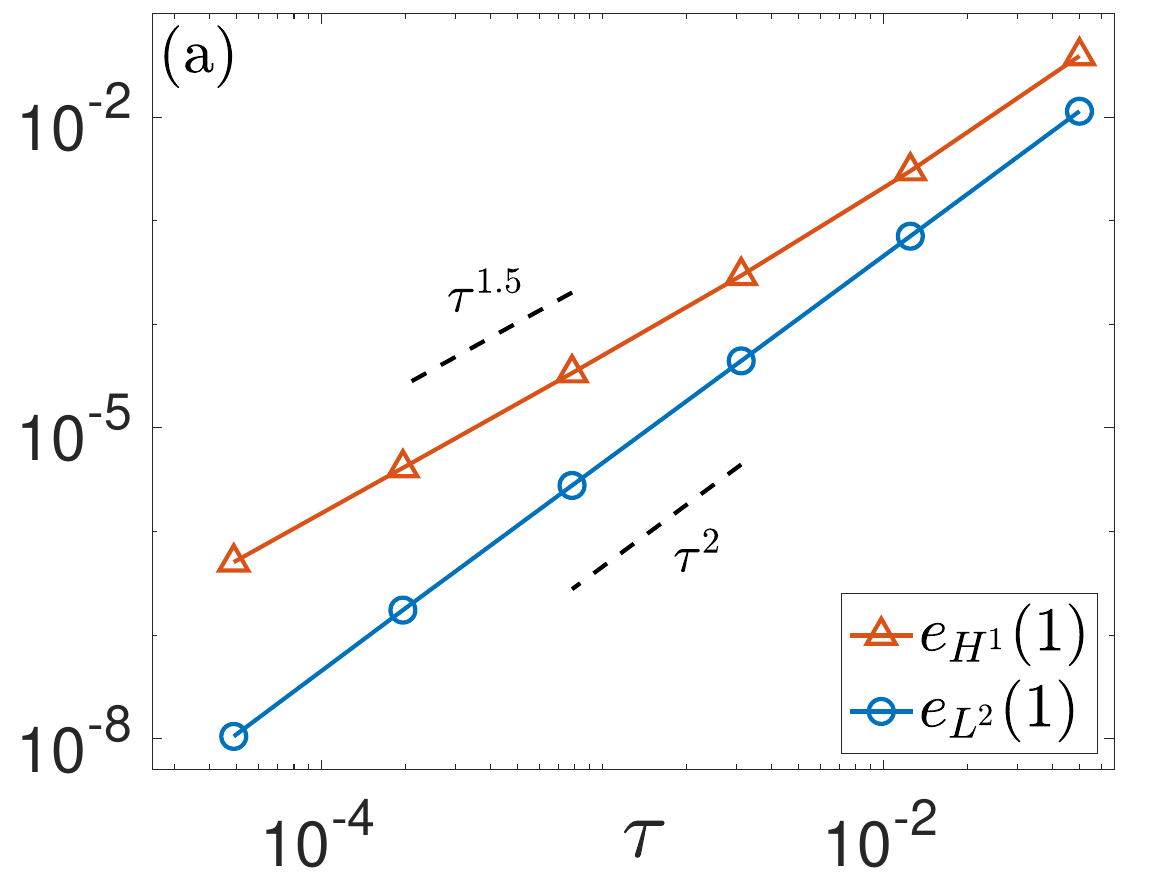}}\hspace{1em}
	{\includegraphics[width=0.475\textwidth]{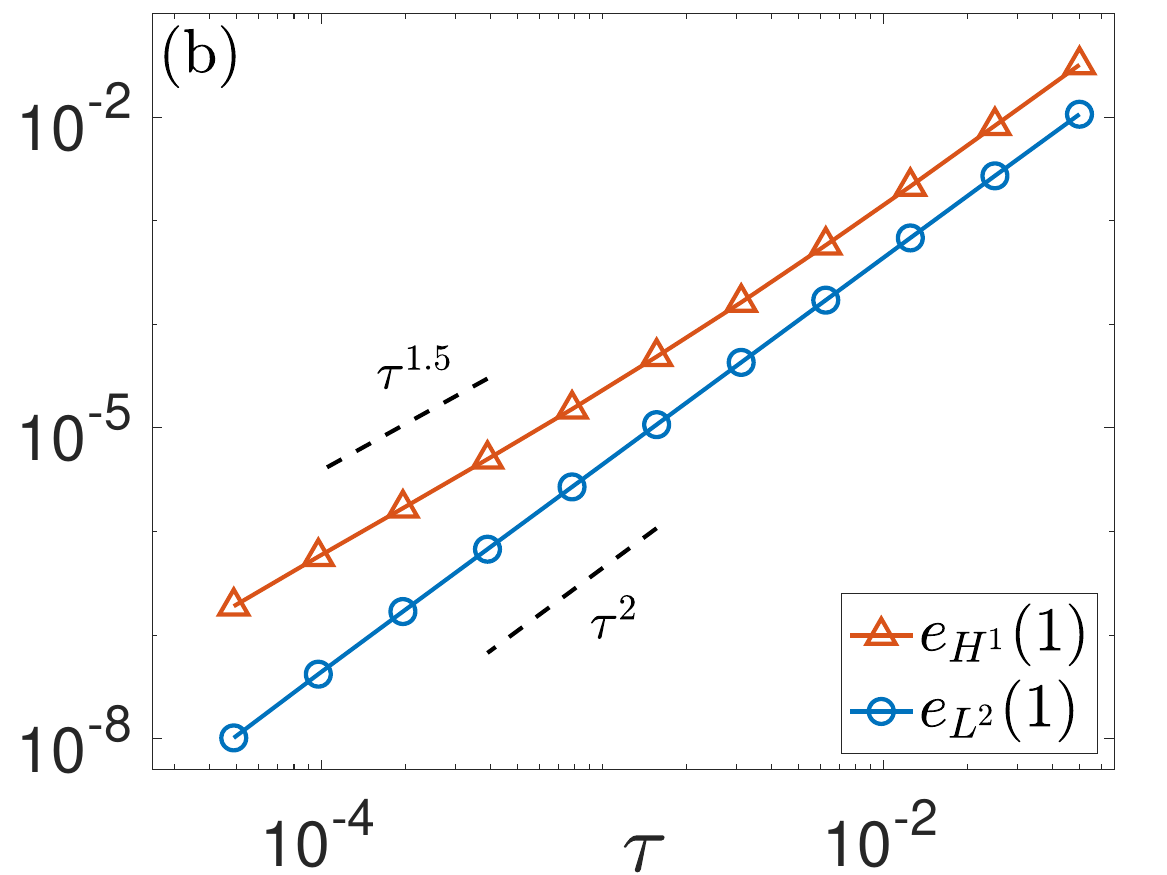}}
	\caption{$L^2$- and $H^1$-errors of the sEWI for the NLSE with $\sigma = 0.5$: (a) errors computed with $h = \sqrt{10\tau}$ and (b) errors computed with $h  = w^{-9}$ fixed.}
	\label{fig:conv_dt_h_H4_ini}
\end{figure}

It should be noted that while all the numerical results shown in \cref{fig:conv_dt_h_H4_ini_good,fig:conv_dt_h_H2_ini_low_reg,fig:conv_dt_h_H4_ini} indicate the sharpness of the convergence rates under given assumptions, they do not imply that all the assumptions are necessary. In particular, the same convergence rate observed in \cref{fig:conv_dt_h_H4_ini} can also be seen when reducing $\sigma$ from $0.5$ to $0.1$. This suggests that the assumption on nonlinearity in \cref{thm:second_order_improved} might be further relaxed. However, the rigorous analysis is challenging and will be considered in our future work. 

In the following, we test the long-time behavior of the sEWI. We consider two cases with potential and nonlinearity of different regularity and compute relative errors of mass and energy up to very long time $T = 500$. In both cases, we fix $\Omega = (-16, 16)$, $\beta = 1$ and choose an odd initial datum $\psi_0(x) = x e^{-x^2/2}$ for $x \in \Omega$. In the case of ``good" potential and nonlinearity, we choose $\sigma = 1.1$ and an $H^2$-potential $V$ given in \cref{eq:good}. In the case of low regularity potential and nonlinearity, we choose $\sigma = 0.1$ and an $L^\infty$-potential $V$ given by
\begin{equation}\label{eq:V_low_1D}
	V(x) = \left\{
	\begin{aligned}
		&10, &&x \geq 4 \text{ or } x \leq -4, \\
		&0, &&\text{otherwise}, 
	\end{aligned}
	\right., \quad x \in \Omega. 
\end{equation}
The numerical results are plotted in \cref{fig:good_long_time,fig:low_reg_long_time} for the two cases respectively. We can clearly observe that the mass and energy are nearly conserved for very long time in both cases: the relative errors of mass and energy satisfy (based on numerical results)
\begin{equation*}
	\frac{|M(\psihn{n}) - M(\psi_0)|}{|M(\psi_0)|} \leq C_1 \tau^2, \quad \frac{|E(\psihn{n}) - E(\psi_0)|}{|E(\psi_0)|} \leq C_2 \tau^2, \quad 0 \leq n\tau \leq T = 500,  
\end{equation*}
where $C_1$ and $C_2$ seem to be independent of $T$. It is rather surprising that even with extremely low regularity potential and nonlinearity, one can still obtain second-order temporal convergence of mass and energy. The rigorous analysis of these phenomena, including the long-time near conservation property and the high-order temporal convergence of mass and energy, will be taken as our future work. 

\begin{figure}[htbp]
	\centering
	{\includegraphics[width=0.475\textwidth]{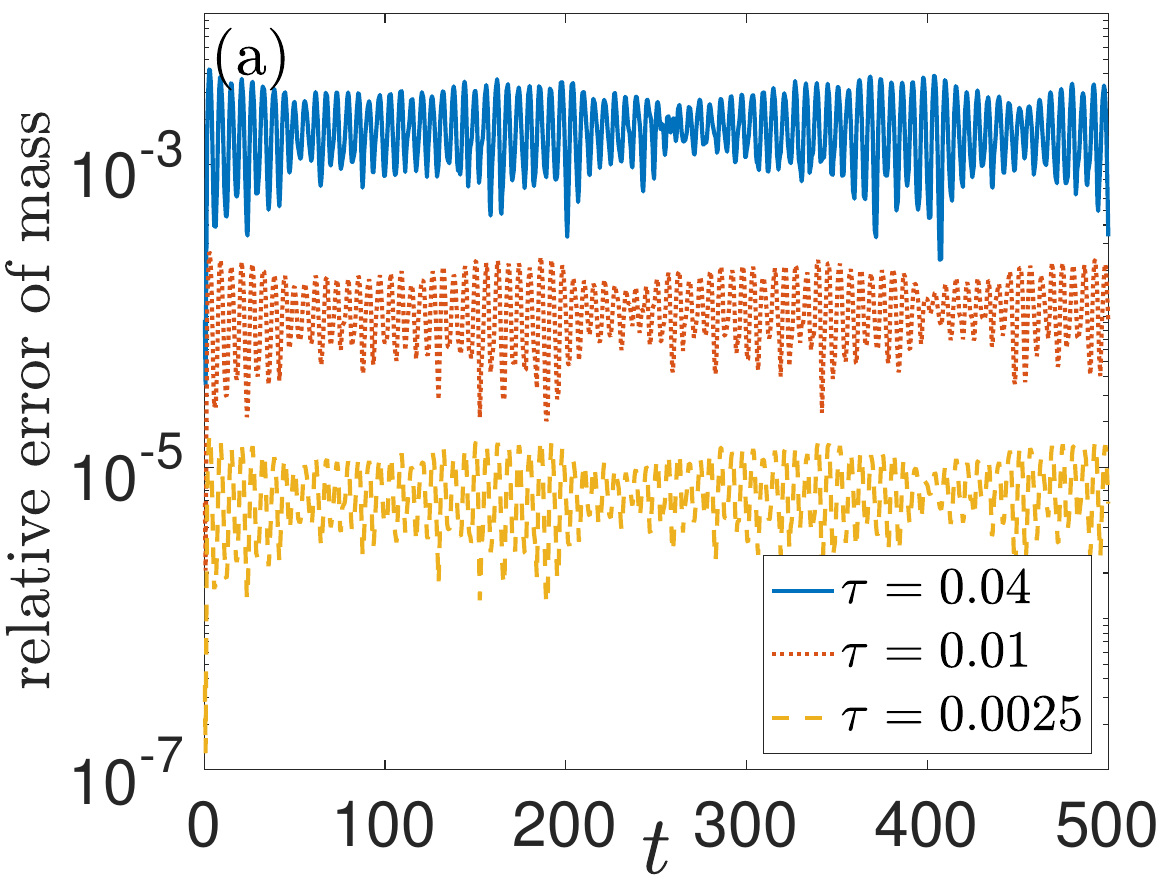}}\hspace{1em}
	{\includegraphics[width=0.475\textwidth]{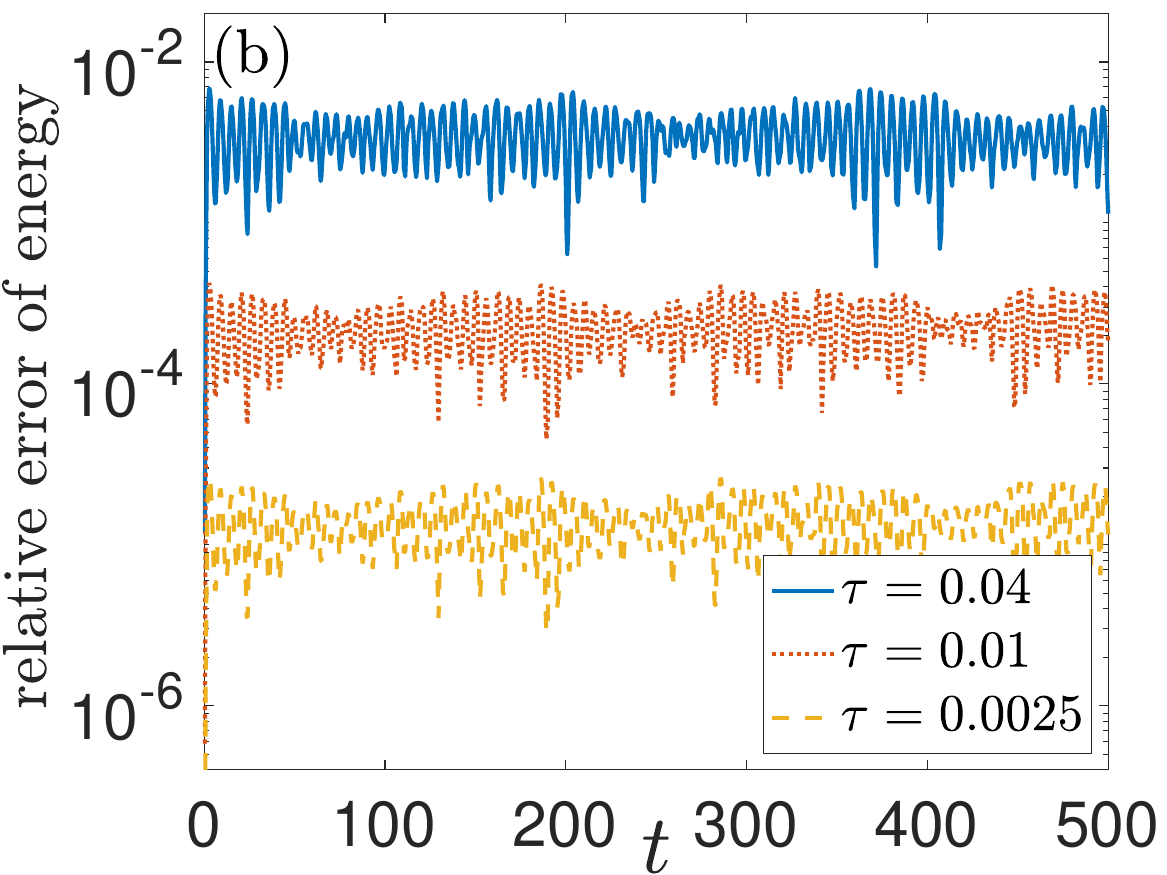}}
	\caption{Relative errors of (a) mass and (b) energy of the sEWI for the NLSE with $\sigma = 1.1$ and $V \in H^2$ given in \cref{eq:good}.}
	\label{fig:low_reg_long_time}
\end{figure}

\begin{figure}[htbp]
	\centering
	{\includegraphics[width=0.475\textwidth]{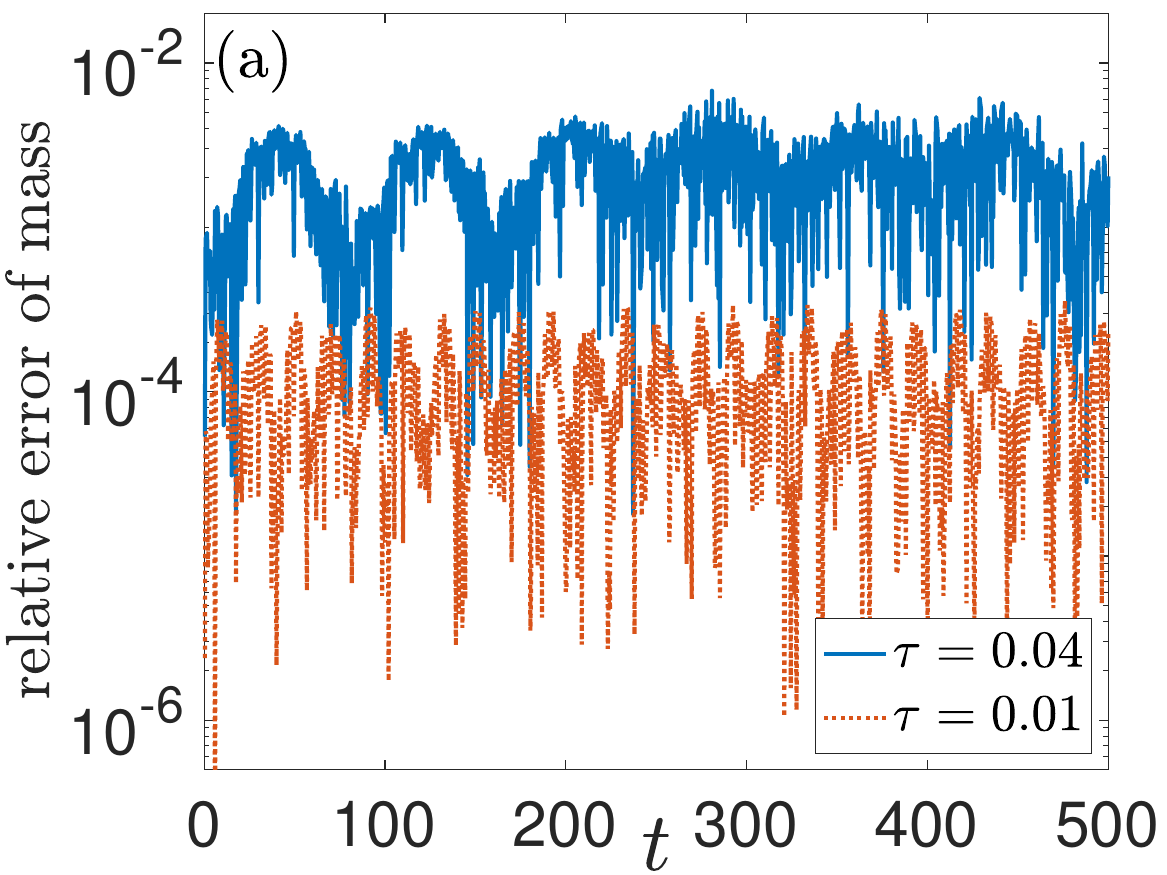}}\hspace{1em}
	{\includegraphics[width=0.475\textwidth]{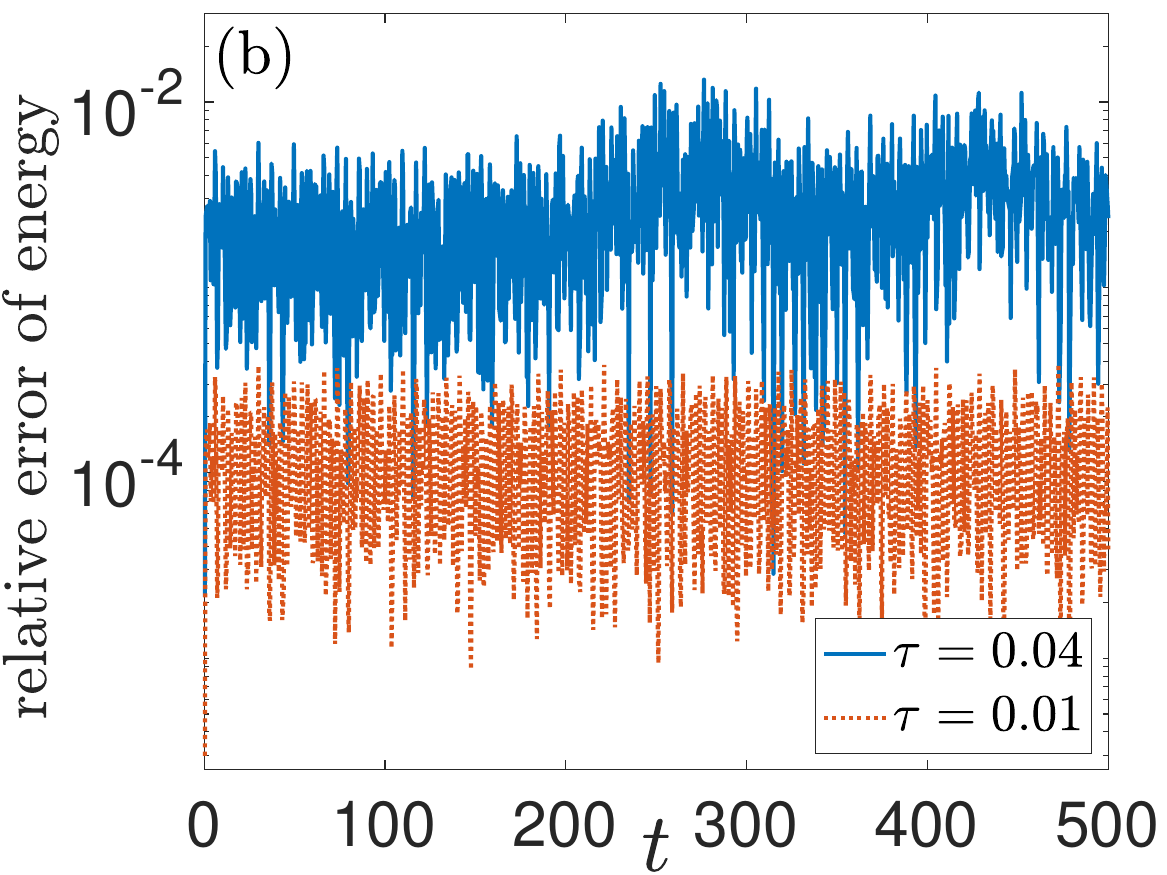}}
	\caption{Relative errors of (a) mass and (b) energy of the sEWI for the NLSE with $\sigma = 0.1$ and $V \in L^\infty$ given in \cref{eq:V_low_1D}.}
	\label{fig:good_long_time}
\end{figure}

Finally, we perform long-time simulation of a benchmark problem posed in \cite{henning_mathcomp} to showcase the superiority of the sEWI even in the smooth setting where the potential, nonlinearity and initial datum are all smooth. In the NLSE \cref{NLSE}, we take $ V(x) \equiv 0 $, $\beta = -2$ and $\sigma = 1$ with $\Omega  = (-16, 16)$. The computational time is taken as $T = 200$ with the initial datum given by
\begin{equation}
	\psi_0(x) = \frac{8(9e^{-4x} + 16e^{4x}) - 32(4 e^{-2x} + 9 e^{2x})}{-128 + 4e^{-6x} + 16e^{6x} + 81e^{-2x} + 64e^{2x}}, \quad x \in \Omega. 
\end{equation}
The exact solution can be given analytically by (24) of \cite{henning_mathcomp}. As discussed in \cite{henning_mathcomp}, although it is a 1D experiment with known analytical solution, the problem is extremely hard to solve numerically. In computation, we take $\tau = 2.5 \times 10^{-6}$ and $h = 2^{-6}$, and the total computational time is about 4 hours on a personal computer. The relative errors of mass and energy as well as the density $|\psi|^2$ are plotted in \cref{fig:benchmark}. We can observe not only mass and energy are near conserved up to long time but also the density is correctly approximated without any evident difference. This example shows the sEWI is still advantageous in the smooth setting to simulate complicated interaction up to long time. 

\begin{figure}[htbp]
	\centering
	{\includegraphics[width=0.475\textwidth]{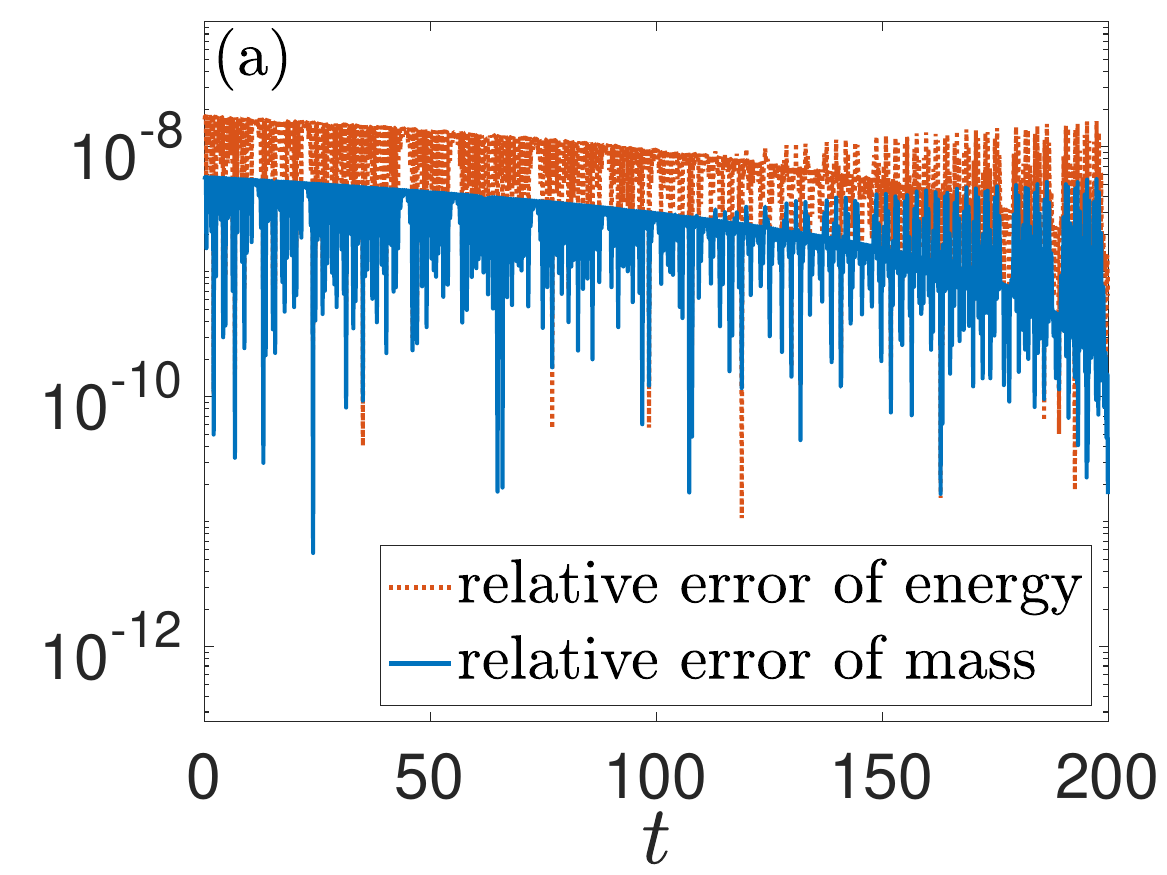}}\hspace{1em}
	{\includegraphics[width=0.475\textwidth]{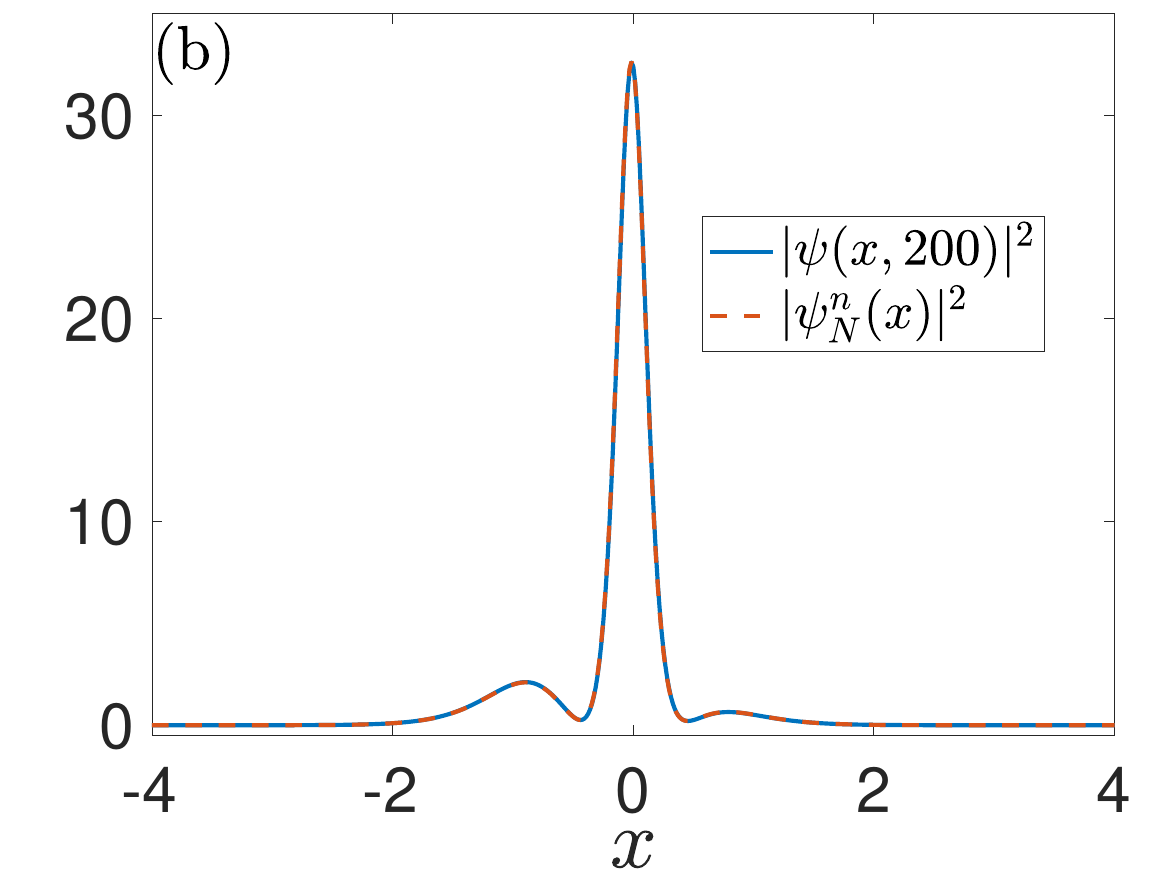}}
	\caption{A benchmark problem: (a) relative errors of mass and energy (b) plots of the density}
	\label{fig:benchmark}
\end{figure}

\section{Conclusion}
We proposed and analyzed an explicit and  symmetric exponential wave integrator (sEWI) for the nonlinear Schr\"odinger equation (NLSE) with low regularity potential and nonlinearity. The sEWI is symmetric, explicit and stable under a time step size restriction independent of the mesh size. Moreover, it exhibits excellent performance in solving the NLSE with low regularity potential and/or nonlinearity as well as in the long-time simulation of the NLSE. Rigorous error estimates of the sEWI were established under various regularity assumptions on potential and nonlinearity. Extended numerical results were reported to validate our error estimates and to demonstrate superb long-time behavior of the sEWI. 

\bibliographystyle{siamplain}

\end{document}